\documentclass[10pt]{amsart}
\usepackage{amssymb}
\usepackage{amscd}

\numberwithin{equation}{section}

\newcounter{TmpEnumi}

\def\today{\number\day\space\ifcase\month\or   January\or February\or
  March\or April\or May\or June\or   July\or August\or September\or
  October\or November\or December\fi\   \number\year}

\theoremstyle{definition}
\newtheorem{thm}{Theorem}[section]
\newtheorem{lem}[thm]{Lemma}
\newtheorem{prp}[thm]{Proposition}
\newtheorem{dfn}[thm]{Definition}
\newtheorem{cor}[thm]{Corollary}
\newtheorem{cnj}[thm]{Conjecture}

\newtheorem{ntn}[thm]{Notation}
\newtheorem{exa}[thm]{Example}

\newtheorem{qst}[thm]{Question}

\newcommand{\beq}{\begin{equation}}
\newcommand{\eeq}{\end{equation}}
\newcommand{\beqr}{\begin{eqnarray*}}
\newcommand{\eeqr}{\end{eqnarray*}}
\newcommand{\bal}{\begin{align*}}
\newcommand{\eal}{\end{align*}}
\newcommand{\bei}{\begin{itemize}}
\newcommand{\eei}{\end{itemize}}
\newcommand{\limi}[1]{\lim_{{#1} \to \infty}}

\newcommand{\af}{\alpha}
\newcommand{\bt}{\beta}
\newcommand{\gm}{\gamma}
\newcommand{\dt}{\delta}
\newcommand{\ep}{\varepsilon}
\newcommand{\zt}{\zeta}
\newcommand{\et}{\eta}
\newcommand{\ch}{\chi}

\newcommand{\ld}{\lambda}
\newcommand{\sm}{\sigma}
\newcommand{\kp}{\kappa}
\newcommand{\ph}{\varphi}
\newcommand{\ps}{\psi}
\newcommand{\rh}{\rho}
\newcommand{\om}{\omega}
\newcommand{\ta}{\tau}

\newcommand{\Gm}{\Gamma}

\newcommand{\Ld}{\Lambda}

\newcommand{\Om}{\Omega}

\newcommand{\Z}{{\mathbb{Z}}}
\newcommand{\R}{{\mathbb{R}}}
\newcommand{\C}{{\mathbb{C}}}
\newcommand{\N}{{\mathbb{Z}}_{> 0}}

\newcommand{\CC}{{\mathcal{C}}}

\newcommand{\id}{{\mathrm{id}}}

\newcommand{\sint}{{\mathrm{int}}}

\newcommand{\Prim}{{\mathrm{Prim}}}

\newcommand{\supp}{{\mathrm{supp}}}

\newcommand{\spn}{{\mathrm{span}}}
\newcommand{\card}{{\mathrm{card}}}
\newcommand{\Aut}{{\mathrm{Aut}}}
\newcommand{\Ad}{{\mathrm{Ad}}}

\newcommand{\dirlim}{\varinjlim}

\newcommand{\andeqn}{\,\,\,\,\,\, {\mbox{and}} \,\,\,\,\,\,}

\newcommand{\Wolog}{Without loss of generality}
\newcommand{\Tfae}{The following are equivalent}
\newcommand{\tfae}{the following are equivalent}
\newcommand{\ifo}{if and only if}

\newcommand{\ca}{C*-algebra}

\newcommand{\hm}{homomorphism}

\newcommand{\hsa}{hereditary subalgebra}

\newcommand{\pj}{projection}
\newcommand{\mops}{mutually orthogonal \pj s}
\newcommand{\nzp}{nonzero projection}

\newcommand{\cfn}{continuous function}

\newcommand{\hme}{homeomorphism}

\renewcommand{\S}{\subset}

\newcommand{\SM}{\setminus}
\newcommand{\I}{\infty}
\newcommand{\E}{\varnothing}

\title[Spectrally free actions]{Crossed products
 by spectrally free actions}

\author{Cornel Pasnicu}
\author{N.~Christopher Phillips}

\date{22~August 2013}

\address{Department of Mathematics,
      The University of Texas at San Antonio,
      San Antonio TX 78249, USA.}
\email[]{Cornel.Pasnicu@utsa.edu}
\address{Department of Mathematics, University of Oregon,
      Eugene OR 97403-1222, USA.}
\email[]{ncp@darkwing.uoregon.edu}

\subjclass{Primary 46L55;
 Secondary 46L35, 46L40.}
\thanks{Some of this material is based upon work of the second
   author supported by the US National Science Foundation
   under Grants DMS-0302401, DMS-0701076, and DMS-1101742.}

\begin{document}

\begin{abstract}
We define spectral freeness
for actions of discrete groups on C*-algebras.
We relate spectral freeness to other freeness conditions;
an example result is that for an action $\alpha$
of a finite group~$G,$
spectral freeness is equivalent to strong pointwise outerness,
and also to the condition that
${\widetilde{\Gamma}} (\alpha_g) \neq \{ 1 \}$
for every $g \in G \setminus \{ 1 \}.$

We then prove permanence results for reduced crossed products
by exact spectrally free actions,
for crossed products
by arbitrary actions of ${\mathbb{Z}} / 2 {\mathbb{Z}},$
and for extensions, direct limits, stable isomorphism,
and several related constructions,
for the following properties:
\begin{itemize}
\item
The combination of pure infiniteness and the ideal property.
\item
Residual hereditary infiniteness (closely related to pure infiniteness).
\item
Residual (SP)
(a strengthening of Property~(SP) suitable for nonsimple C*-algebras).
\item
The weak ideal property (closely related to the ideal property).
\end{itemize}
For the weak ideal property,
we can allow arbitrary crossed products by any finite abelian group.

These properties of C*-algebras are shown to
have formulations of the same general type,
allowing them all to be handled using a common set of theorems.
\end{abstract}

\maketitle

\indent
We prove permanence results for
crossed products by exact spectrally free actions.
(We say more about spectral freeness below.)
We also give similar results for completely arbitrary actions
of the group $\Z_2.$
(In fact, if one of the properties we consider holds for the
fixed point algebra of an action of $\Z_2$ on a \ca~$A,$
then it holds for~$A.$)
For the most part,
we consider the following properties,
either already known or related to known properties,
about which we say more below:
\begin{itemize}
\item
Residual~(SP).
\item
Residual hereditary (proper) infiniteness.
\item
The combination of pure infiniteness (for nonsimple C*-algebras)
and the ideal property.
\item
The weak ideal property.
\end{itemize}
In fact, we show that if $\af \colon G \to \Aut (A)$
is an action of an arbitrary finite group,
and if the fixed point algebra has the weak ideal property,
then $A$ has the weak ideal property.

The best plausible related permanence results
are that
crossed products by arbitrary discrete groups
should preserve pure infiniteness
(not just residual hereditary infiniteness),
while crossed products by exact actions of discrete groups
should preserve residual~(SP) and the ideal property when,
except for a finite normal subgroup,
the action satisfies a suitable outerness condition.
We do not give theorems in anything like this generality.
For example,
it remains unknown whether crossed products by arbitrary actions
of finite groups preserve
the ideal property or pure infiniteness.

Our methods also yield easy proofs
of some other permanence results,
using a general scheme given in Section~\ref{Sec_HSAPerm}.

The outerness condition we use is
spectral freeness.
It seems to be an appropriate version for
nonsimple \ca{s} of pointwise outerness;
pointwise outerness is too weak a hypothesis to be able to prove much.
It is based on the notion of a freely acting automorphism
from Section~2 of~\cite{Ks7}.
It seems to be a better hypothesis for theorems
than strong pointwise outerness
(\cite{Phfgs}; see Definition~\ref{D_3718_SPOut} below.)
For finite groups, but not in general,
spectral freeness is equivalent to strong pointwise outerness
(Theorem~\ref{T_3725_FgSpF}; Proposition~\ref{P-608-SFNotSPO}).

A \ca{} is hereditarily infinite if every
every nonzero hereditary subalgebra contains
an infinite positive element
(in the sense of Definition~3.2 of~\cite{KR}),
and residually hereditarily infinite
if every quotient is hereditarily infinite.
This property appears (without a name) in~\cite{KR},
where it is shown to follow from pure infiniteness
(for nonsimple C*-algebras).
Question~4.8 of~\cite{KR} asks whether the converse holds;
this question seems to be still open.
We have few permanence results
for pure infiniteness of crossed products;
permanence results for residual hereditary infiniteness
are what we can prove instead,
and are suggestive about permanence results
for pure infiniteness.

The combination of pure infiniteness and the ideal property
is an interesting condition in its own right,
and permanence results for it are suggestive of
permanence results for both properties separately.

Recall that a \ca{} has Property~(SP)
if every nonzero hereditary subalgebra has a nonzero \pj.
This property has mostly been used
for simple \ca{s},
and seems not to be the right property for nonsimple \ca{s}.
In Example~\ref{E-QSPvsSP},
we give a \ca{} which has Property~(SP) but which has quotients
which do not have Property~(SP).
Residual~(SP) is the requirement that every quotient algebra
has Property~(SP),
and seems better for nonsimple algebras.

The weak ideal property is a weakening of the ideal property.
Instead of requiring that every ideal
be generated by its \pj{s},
we require that every nonzero subquotient
of the stabilization contain a nonzero \pj.
This condition admits better permanence results
(unrelated to crossed products)
than the ideal property does.
For example, it is preserved by extensions;
it is known (Theorem~5.1 of~\cite{Psn1}) that the ideal property is not.
Permanence results for the weak ideal property
for crossed products
are suggestive of permanence results for the ideal property,
although some results for the weak ideal property
are known to fail for the ideal property.

This paper is conceptually related to~\cite{PsPh},
but has different emphasis.
The paper~\cite{PsPh} considered crossed products
and fixed point algebras of actions of finite groups
(with or without freeness conditions)
and a different,
but related, collection of properties.
Here we consider crossed products by actions of infinite groups,
with freeness conditions,
and also give results for crossed products by actions of $\Z_2$
(and a few results for more general finite groups),
proved since~\cite{PsPh} was written.
We presume our results for $\Z_2$
generalize to arbitrary finite groups,
but the generalizations mostly seem to be much harder.
(See the discussion in Section~\ref{Sec:Start}.)

After the work for this paper was done
(but while we were still trying to improve the results for $\Z_2$
to more general finite groups),
the paper~\cite{GrSr} was posted on the arXiv.
It has some overlap with our material on
exact spectrally free actions
and crossed products specifically by actions on
purely infinite \ca{s} with the ideal property.
It does everything in the more general context of partial actions,
but does not address the other conditions for which we consider
permanence results,
and does not address general actions of~$\Z_2.$
Residual topological freeness,
as defined in~\cite{GrSr},
is, when restricted to actions on commutative \ca{s},
the same as our spectral freeness.
(See Proposition~\ref{P_3720_RTF} below.)
Our Lemma~\ref{L-L8_012Mod} is related to Lemma~3.11 of~\cite{GrSr}
and has a similar proof,
but has differently stated hypotheses and conclusion.
Our Theorem \ref{T_3720_PIIP}(\ref{T_3720_PIIP_SpFr})
is Theorem~4.2 of~\cite{GrSr} for actions.

This paper is organized as follows.
In Section~\ref{Sec:SNA},
we define spectral freeness and relate it
to other properties which have previously been considered.
For example, for finite groups but not in general,
spectral freeness of an action $\af \colon G \to \Aut (A)$
is equivalent to strong pointwise outerness
and to the condition that
${\widetilde{\Gm}} (\af_g) \neq \{ 1 \}$
for every $g \in G \SM \{ 1 \}.$

In Section~\ref{Sec:ZRP},
we consider actions of~$\Z$ with the Rokhlin property.
We prove an averaging lemma,
and, as consequences,
we show that such actions are spectrally free,
their crossed products
preserve the projection property
(Definition~1 of~\cite{Psn})
and the ideal property,
and that every ideal in the crossed product is a crossed
product by an invariant ideal
(a result attributed to us in a 2006 preprint~\cite{Ln}).

Section~\ref{Sec:UsingSpFree}
contains the main technical result on spectrally free actions:
every nonzero hereditary subalgebra of the reduced crossed product
contains an isomorphic image
of a nonzero hereditary subalgebra of the original algebra.
In Section~\ref{Sec:Start},
we prove this for arbitrary actions of~$\Z_2.$
In fact,
for actions of~$\Z_2,$
every nonzero hereditary subalgebra of the given algebra
contains an isomorphic image
of a nonzero hereditary subalgebra of the fixed point algebra.

Section~\ref{Sec_HSAPerm} gives permanence results for
a class of properties of \ca{s} defined using hereditary
subalgebras.
In the remaining three sections,
we use these results to treat the properties discussed earlier
in the introduction.
Residual hereditary (proper) infiniteness
and the combination of pure infiniteness and the ideal property
are treated
in Section~\ref{Sec:PermHInf}.
Residual~(SP) is treated in Section~\ref{Sec:PermQSP}.
The weak ideal property is treated in Section~\ref{Sec:WIP},
along with the intermediate condition
that every nonzero subquotient of $A$
contain a nonzero \pj.
This last property
is not of the form required in Section~\ref{Sec_HSAPerm},
but we can still prove some permanence results for it.
For both this property and the weak ideal property,
we are able to prove preservation by crossed products by
arbitrary finite abelian groups (not just~$\Z_2$).

The crossed product results are derived
from results involving fixed point algebras.
What we actually show is that if $\af \colon \Z_2 \to \Aut (A)$
is any action,
and if the fixed point algebra has the given property,
then so does~$A.$
For the weak ideal property and the intermediate condition above,
we can allow an arbitrary finite group in place of~$\Z_2.$
The crossed product results follow by duality.
This is why we require the group to be abelian
when considering crossed products of algebras
with the weak ideal property.

We use the following conventions and notation,
some standard
(some of them recalled here for reference)
and some less so.

If $A$ is a \ca,
then $A_{+}$ denotes the set of positive elements of~$A.$
As usual, we let $K$ denote the C*-algebra of compact operators on a
separable infinite dimensional Hilbert space.
We set $\Z_n = \Z / n \Z.$
(The $p$-adic integers will not appear.)

We write $a \sim b$ for Cuntz equivalence and $a \precsim b$
for Cuntz subequivalence.
See Definition~2.1 of~\cite{KR}
(except that $a \sim b$ is written $a \approx b$ there),
and see Section~2 of~\cite{KR} for much more on these relations.

For a \ca~$A,$ an automorphism $\ph \in \Aut (A),$
and a subset $S \subset A,$
when we say that $S$ is $\ph$-invariant,
we mean that $\ph (S) = S,$
not merely that $\ph (S) \subset S.$
A $\ph$-invariant quotient of~$A$
means a quotient $A / I$ for a $\ph$-invariant ideal $I \subset A,$
and a $\ph$-invariant subquotient of~$A$
means a quotient $J / I$ for $\ph$-invariant ideals $I, J \subset A$
with $I \subset J.$
The automorphism $\ph$ induces an automorphism of
any invariant subalgebra, quotient, or subquotient.

For an action $\af \colon G \to \Aut (A)$
of a group $G$ on a \ca~$A,$
we give the obvious analogous meaning to the terms
``$\af$-invariant quotient'' and ``$\af$-invariant subquotient'',
and note that $\af$ induces actions on such quotients and subquotients.
Thus, for example, a $\ph$-invariant quotient
is a quotient which is invariant for the action of $\Z$ on~$A$
generated by~$\ph.$

We sometimes use the same symbol~$\af$
for the action
$g \mapsto \af_{g} |_I$
induced by $\af$ on an invariant ideal $I \subset A,$
and similarly with invariant subalgebras and subquotients,
as well as $M_n (A)$ and similar constructions.
We denote the fixed point algebra by~$A^{\af}.$

We use the Connes spectrum and Borchers spectrum
of an action $\af \colon G \to \Aut (A)$
of a locally compact abelian group $G$ on a \ca~$A,$
denoted $\Gm (\af)$ and $\Gm_{\mathrm{B}} (\alpha).$
They are defined, for example,
at the beginning of Section~1 of~\cite{Ks7}
(not the original source).
We also use the strong Connes spectrum ${\widetilde{\Gm}} (\alpha),$
as defined in~\cite{Ks1}.
If $\varphi$ is an automorphism of a \ca~$B,$
and $\bt \colon \Z \to \Aut (B)$ is the action generated by~$\ph,$
we write $\Gm (\varphi)$ for $\Gm (\bt),$
and similarly
$\Gm_{\mathrm{B}} (\varphi)$ for $\Gm_{\mathrm{B}} (\bt)$
and ${\widetilde{\Gm}} (\varphi)$ for ${\widetilde{\Gm}} (\bt).$


The following definition is taken from the introduction of~\cite{Sr},
where it is applied to the situation in which $B$ is the reduced
crossed product by an action on~$A.$

\begin{dfn}\label{D_3718_SepId}
Let $B$ be a \ca,
and let $A \S B$ be a subalgebra.
We say that $A$ {\emph{separates the ideals of~$B$}}
if whenever $I, J \S B$ are ideals such that $I \cap A = J \cap A,$
then $I = J.$
\end{dfn}

Some of this work was done during visits by the second author to
K{\o}benhavns Universitet during March--May 2012
and to the University of Texas at San Antonio
and Tokyo University during November and December 2012.
He is grateful to those institutions for their hospitality.

\section{Spectrally free actions}\label{Sec:SNA}

\indent
In this section,
we motivate and define spectral freeness,
the outerness condition we use,
and relate it to other conditions considered previously,
particularly when the group is finite,
the algebra is commutative,
or the algebra is simple.
Consequences of spectral freeness
will be given in Section~\ref{Sec:UsingSpFree}.

For many purposes,
when $A$ is simple,
the right version of freeness for an action $\af \colon G \to \Aut (A)$
of a discrete group~$G$
is pointwise outerness:
for all $g \in G \SM \{ 1 \},$
the automorphism $\af_g$ is not implemented by a unitary
in the multiplier algebra $M (A).$
For example, $C^*_{\mathrm{r}} (G, A, \af)$ is again simple
(Theorem~3.1 of~\cite{Ks2}).
For nonsimple algebras,
this condition is far too weak to be useful.
(Consider an action on a direct sum which is inner on one
summand and pointwise outer on the other.)
In Section~4 of~\cite{Phfgs},
the condition in the following definition
was implicitly advocated as a substitute,
at least for finite groups.

\begin{dfn}[Definition~4.11 of~\cite{Phfgs}]\label{D_3718_SPOut}
Let $A$ be a \ca{} and let $G$ be a group.
An action $\af \colon G \to \Aut (A)$ is said to be
{\emph{strongly pointwise outer}}
if, for every $g \in G \setminus \{ 1 \}$
and any two $\af_g$-invariant ideals $I \subset J \subset A$
with $I \neq J,$
the automorphism of $J / I$ induced by $\af_g$ is outer,
that is, not of the form $a \mapsto \Ad (u) (a) = u a u^*$
for any unitary $u$ in the multiplier algebra $M (J / I).$
\end{dfn}

For finite group actions,
some justification for this condition
is given in Theorem~4.12 of~\cite{Phfgs},
and Examples 4.13 and~4.14 of~\cite{Phfgs}
show that several obvious weaker versions are not suitable.
Proposition~\ref{P-608-SFNotSPO} below,
however, suggests that strong pointwise outerness
is too strong.

We give a preliminary definition,
from the beginning of Section~2 of~\cite{Ks7},
where $\ph \in \Aut (A)$ is said to
be a ``freely acting automorphism'' if the condition is satisfied.
We don't use that term,
because it becomes awkward when applied to group actions.

\begin{dfn}\label{D-FrActAut403}
Let $A$ be a \ca,
and let $\ph \in \Aut (A).$
Then $\ph$ is said to be {\emph{spectrally nontrivial}}
if for every nonzero ideal $I \subset A$ such that $\ph (I) = I,$
we have $\Gm_{\mathrm{B}} (\varphi |_I) \neq \{ 1 \}$
(as a subset of ${\widehat{\Z}} = S^1$).
Otherwise, we say that $\ph$ is {\emph{spectrally trivial.}}
\end{dfn}

We generalize this definition as follows.

\begin{dfn}\label{D-FrActGp403}
Let $\af \colon G \to \Aut (A)$ be an action of a discrete group $G$
on a \ca~$A.$
We say that $\af$ is {\emph{pointwise spectrally nontrivial}}
if for every $g \in G \setminus \{ 1 \},$
the automorphism $\af_g$ is spectrally nontrivial in the sense
of Definition~\ref{D-FrActAut403}.

We further say that $\af$
is {\emph{spectrally free}}
if for any $\alpha$-invariant ideal $I$ of $A$ such
that $I \neq A,$
the induced action on $A / I$ is pointwise spectrally nontrivial.
\end{dfn}

In the rest of this section,
we give various results which support the idea
that spectral freeness is a good form of noncommutative freeness.

Spectral freeness
should be thought of as a substitute
for strong pointwise outerness.
When the algebra is simple,
spectral freeness and strong pointwise outerness
reduce to pointwise spectral nontriviality and pointwise outerness.
The following fact is essentially immediate from
what is already known.

\begin{prp}\label{P-ASimp608}
Let $\af \colon G \to \Aut (A)$ be an action of a discrete group~$G$
on a simple \ca~$A.$
Then $\af$ is pointwise spectrally nontrivial
\ifo\  $\af$ is pointwise outer.
\end{prp}

\begin{proof}
When $A$ is simple, for every $\ph \in \Aut (A)$ the definitions
immediately imply that $\Gm_{\mathrm{B}} (\varphi) = \Gm (\varphi).$
But according to Corollary 8.9.10 of~\cite{Pd1},
outerness of~$\ph$ is equivalent to $\Gm (\varphi) \neq \{ 1 \}.$
\end{proof}

Although it is not directly related to spectral freeness,
one direction of Proposition~\ref{P-ASimp608} is true in general.

\begin{prp}\label{P-SFree608}
Let $\af \colon G \to \Aut (A)$ be a
pointwise spectrally nontrivial action of a discrete group~$G$
on a \ca~$A.$
Then $\af$ is pointwise outer.
\end{prp}

\begin{proof}
It suffices to prove that if $\ph \in \Aut (A)$
is spectrally nontrivial,
then $\ph$ is outer.
This follows from Theorem~2.1 of~\cite{Ks7}.
\end{proof}

We can also relate spectral freeness to the strong Connes spectrum.
We need a preliminary result on the strong Connes spectrum
itself.

\begin{prp}\label{P_3723_CSpSbQ}
Let $\af \colon G \to \Aut (A)$ be an action
of a locally compact abelian group~$G$
on a \ca~$A,$
and let $I_1 \S I_2 \S A$ be $\af$-invariant ideals.
Let ${\overline{\af}}$ be the action of $G$ on $I_2 / I_1$
determined by~$\af.$
Then
${\widetilde{\Gm}} (\af) \S {\widetilde{\Gm}} ( {\overline{\af}} ).$
\end{prp}

\begin{proof}
Set $B = C^* (G, A, \af),$
$J_1 = C^* (G, I_1, \af),$
and $J_2 = C^* (G, I_2, \af).$
Lemma~2.8.2 of~\cite{Ph1} allows us to
identify $C^* (G, \, I_2 / I_1, \, {\overline{\af}} )$
with $J_2 / J_1,$
and this identification clearly respects the dual actions
$\bt \colon {\widehat{G}} \to \Aut (B)$
and ${\overline{\bt}} \colon {\widehat{G}} \to \Aut (J_2 / J_1).$
Lemma~3.4 of~\cite{Ks2} gives
\[
{\widetilde{\Gm}} (\af)
 = \big\{ \ta \in {\widehat{G}} \colon
   {\mbox{$\bt_{\ta} (L) \S L$ for every ideal $L \S B$}} \big\}
\]
and
\[
{\widetilde{\Gm}} ({\overline{\af}})
 = \big\{ \ta \in {\widehat{G}} \colon
   {\mbox{${\overline{\bt}}_{\ta} (L) \S L$
    for every ideal $L \S J_2 / J_1$}} \big\}.
\]
It is clear from the compatibility of the actions
$\bt$ and~${\overline{\bt}}$ that
${\widetilde{\Gm}} (\af) \S {\widetilde{\Gm}} ( {\overline{\af}} ).$
\end{proof}

\begin{prp}\label{P_3723_StCSpToSpFr}
Let $\af \colon G \to \Aut (A)$ be an action of a discrete group~$G$
on a \ca~$A.$
Suppose that
${\widetilde{\Gm}} (\af_g) \neq \{ 1 \}$
for every $g \in G \SM \{ 1 \}.$
Then $\af$ is spectrally free.
\end{prp}

\begin{proof}
Let $g \in G \SM \{ 1 \},$
let $I \S A$ be an $\af$-invariant ideal,
let ${\overline{\af}} \colon G \to A / I$
be the action determined by~$\af,$
and let $J \S A / I$ be an ${\overline{\af}}_g$-invariant ideal.
We have to show that
$\Gm_{\mathrm{B}} ( {\overline{\af}}_g |_J) \neq \{ 1 \}.$
By hypothesis, ${\widetilde{\Gm}} (\af_g) \neq \{ 1 \}.$
Applying Proposition~\ref{P_3723_CSpSbQ}
with $\Z$ in place of~$G,$
we obtain ${\widetilde{\Gm}} ({\overline{\af}}_g |_J) \neq \{ 1 \}.$
The result now follows from
${\widetilde{\Gm}} ({\overline{\af}}_g |_J)
   \S \Gm_{\mathrm{B}} ( {\overline{\af}}_g |_J).$
\end{proof}

In general,
spectral freeness implies neither
${\widetilde{\Gm}} (\af_g) \neq \{ 1 \}$
for every $g \in G \SM \{ 1 \}$
nor strong pointwise outerness.
There are counterexamples in the commutative case,
which we consider next.
(These conditions are all equivalent for finite groups.
See Theorem~\ref{T_3725_FgSpF} below.)

\begin{lem}\label{L-608HmeSNT}
Let $X$ be a locally compact Hausdorff space,
and let $h \colon X \to X$ be a \hme.
Let $\ph \in \Aut (C_0 (X))$ be the automorphism
given by $\ph (f) = f \circ h^{-1}$ for $f \in C_0 (X).$
Then $\ph$ is spectrally nontrivial \ifo\  %
\[
\sint \big( \{ x \in X \colon h (x) = x \} \big) = \E.
\]
\end{lem}

\begin{proof}
Set
$F = \{ x \in X \colon h (x) = x \}.$

Assume that $\sint (F) = \E.$
We use Theorem~2.1 of~\cite{Ks7}
to show that $\ph$ is spectrally nontrivial.
Thus,
let $B$ be a nonzero hereditary subalgebra of $C (X),$
and let $f_0 \in C (X).$
We have to show that for every $\ep > 0$
there is $f \in B_{+}$ such that $\| f \| = 1$
and $\| f f_0 \ph (f) \| < \ep.$
There is a nonempty open set $U \subset X$ such that
\[
B = \big\{ f \in C (X) \colon
  {\mbox{$f (y) = 0$ for all $y \not\in U$}} \big\}.
\]
Then $U \not\subset F,$
so there is $x \in U$ with $h (x) \neq x.$
Choose an open set $V \subset U$ containing $x$
such that $h (V) \cap V = \E.$
Choose a \cfn\  $f \colon X \to [0, 1]$
such that $f (x) = 1$ and $\supp (f) \subset V.$
Then $f \ph (f) = 0,$
so $\| f f_0 \ph (f) \| = 0 < \ep.$
This completes the proof of spectral nontriviality.

Suppose now that $\sint (F) \neq \E.$
Set
\[
I = \big\{ f \in C (X) \colon
  {\mbox{$f (y) = 0$ for all $y \not\in \sint (F)$}} \big\}.
\]
Then $I$ is a $\ph$-invariant ideal of $C_0 (X),$
and $\ph |_I = \id_I,$
so $\Gm_{\mathrm{B}} (\ph) = \{ 1 \}.$
\end{proof}

\begin{prp}\label{P-608CommSF}
Let $X$ be a locally compact Hausdorff space,
and let $(g, x) \mapsto g x$
be an action of a discrete group $G$ on~$X.$
Let $\af \colon G \to \Aut (C_0 (X))$ be the action
$\af_g (f) (x) = f (g^{-1} x)$ for $x \in X$ and $f \in C_0 (X).$
Then:
\begin{enumerate}
\item\label{P-608CommSF-SNT}
$\af$ is pointwise spectrally nontrivial \ifo\  %
for every $g \in G \setminus \{ 1 \},$
the set $\{ x \in X \colon g x = x \}$
has empty interior.
\item\label{P-608CommSF-Spf}
$\af$ is spectrally free \ifo\  %
for every closed $G$-invariant subset $L \subset X$
and every $g \in G \setminus \{ 1 \},$
the set $\{ x \in L \colon g x = x \}$
has empty interior in~$L.$
\item\label{P-608AbOnComm}
If $G$ is abelian,
then $\af$ is spectrally free \ifo\  %
the action of $G$ on $X$ is free.
\end{enumerate}
\end{prp}

\begin{proof}
Part~(\ref{P-608CommSF-SNT}) is immediate from Lemma~\ref{L-608HmeSNT}.
Part~(\ref{P-608CommSF-Spf})
now follows from the fact that the $G$-invariant quotients of $C_0 (X)$
are exactly the algebras $C_0 (L)$
for closed $G$-invariant subsets $L \subset X.$

We prove~(\ref{P-608AbOnComm}).
If the action is free, then $\af$ is spectrally free
by part~(\ref{P-608CommSF-Spf}).

Now suppose that the action is not free.
Choose $g_0 \in G \setminus \{ 1 \}$ and $x_0 \in X$
such that $g_0 x_0 = x_0.$
Set $L = {\overline{G x_0}}.$
We claim that $g_0 x = x$ for all $x \in L.$
It suffices to consider $x = h x_0$ with $h \in G,$
and we have $g_0 h x_0 = h g_0 x_0 = h x_0$
because $G$ is abelian.
The claim follows.
Apply part~(\ref{P-608CommSF-Spf})
with this choice of $L$ to see that $\af$ is not spectrally free.
\end{proof}

Residual topological freeness is introduced
in Definition 3.4(ii) of~\cite{GrSr}.
An action of a group $G$ on a Hausdorff topological space~$X$
is essentially free if
for $g \in G \setminus \{ 1 \},$
the set $\{ x \in X \colon g x = x \}$
has empty interior.
On a commutative \ca,
one can check that residual topological freeness
is equivalent to essential freeness
of the action on every $G$-invariant closed set
in the corresponding topological space.

\begin{prp}\label{P_3720_RTF}
Let $A$ be a commutative \ca,
let $G$ be a discrete group,
and let $\af \colon G \to \Aut (A)$ be an action of $G$ on~$A.$
Then $\af$ is spectrally free
\ifo{} $\af$ is residually topologically free.
\end{prp}

\begin{proof}
Apply Proposition \ref{P-608CommSF}(\ref{P-608CommSF-Spf})
and the discussion above.
\end{proof}

\begin{prp}\label{P-608-SFNotSPO}
Let $G$ be a discrete group,
and let $X$ be a compact metric space with an action of~$G$
which is
minimal,
essentially free
but not free.
Then the corresponding action $\af \colon G \to \Aut (C (X))$
is spectrally free,
but is not strongly pointwise outer
and does not satisfy
${\widetilde{\Gm}} (\af_g) \neq \{ 1 \}$
for every $g \in G \SM \{ 1 \}.$
\end{prp}

There are many such actions
(although none when $G$ is abelian).
Here is an easy example.
Let $G = \Z \rtimes \Z_2$
be the semidirect product of $\Z$ by the automorphism
$n \mapsto - n,$
acting on $S^1$ as follows.
The generator of $\Z$ acts by an irrational rotation,
and the generator of $\Z_2$ acts by $\zt \mapsto \zt^{-1}.$

\begin{proof}[Proof of Proposition~\ref{P-608-SFNotSPO}]
Spectral freeness follows from
Proposition \ref{P-608CommSF}(\ref{P-608CommSF-Spf}).

We now prove that $\af$ is not strongly pointwise outer
and that there is $g \in G \SM \{ 1 \}$
such that ${\widetilde{\Gm}} (\af_g) = \{ 1 \}.$
By hypothesis,
there exist $g \in G \setminus \{ 1 \}$ and $x \in X$
such that $g x = x.$
Then
\[
C (X) / \{ f \in C (X) \colon f (x) = 0 \}
\]
is a nonzero $\af_g$-invariant subquotient
on which the automorphism $\bt$ induced by $\af_g$
is trivial,
hence not outer.
Clearly also ${\widetilde{\Gm}} (\bt) = \{ 1 \}.$
So ${\widetilde{\Gm}} (\af_g) = \{ 1 \}$
by Proposition~\ref{P_3723_CSpSbQ}.
\end{proof}

We now show the equivalence,
for actions of finite groups,
of spectral freeness,
strong pointwise outerness,
and ${\widetilde{\Gm}} (\af_g) \neq \{ 1 \}$
for every $g \in G \SM \{ 1 \}.$

\begin{prp}\label{P_3723_SPOtoStCsp}
Let $\af \colon G \to \Aut (A)$
be a strongly pointwise outer action
of a finite group~$G$ on a \ca~$A.$
Then ${\widetilde{\Gm}} (\af_g) \neq \{ 1 \}$
for every $g \in G \SM \{ 1 \}.$
\end{prp}

\begin{proof}
Let $g \in G.$
It is obvious that the restriction of any
strongly pointwise outer action to any subgroup
is again strongly pointwise outer.
Therefore we may assume that $g$ generates~$G.$
Corollary~2.4 of~\cite{PsPh}
implies that ${\widetilde{\Gm}} (\af) = {\widehat{G}}.$
We must show that this implies ${\widetilde{\Gm}} (\af_g) \neq \{ 1 \}.$
The group ${\widehat{G}}$
acts on the primitive ideal space $\Prim ( C^* (G, A, \af) )$
via $\ta \cdot P = {\widehat{\af}}_{\ta} (P)$
for $\ta \in {\widehat{G}}$ and $P \in \Prim ( C^* (G, A, \af) ).$
Lemma~3.4 of~\cite{Ks2}
implies that
${\widehat{\af}}_{\ta} (L) \S L$
for every ideal $L \S C^* (G, A, \af)$
and every $\ta \in {\widehat{G}}.$
Since ${\widehat{G}}$ is finite,
this is equivalent to
${\widehat{\af}}_{\ta} (L) = L$
for every ideal $L \S C^* (G, A, \af)$
and every $\ta \in {\widehat{G}}.$
So the action of ${\widehat{G}}$
on $\Prim ( C^* (G, A, \af) )$ is trivial.

Let $n$ be the order of~$g.$
The surjection $\Z \to G,$ sending $1 \in \Z$ to~$g,$
allows us to identify ${\widehat{G}}$
with the subgroup of ${\widehat{\Z}} = S^1$
consisting of all $\exp (2 \pi i k / n)$ for $k = 0, 1, \ldots, n.$
We now use Corollary~2.5 of~\cite{OPd0},
which identifies $C^* (\Z, A, \af_g)$ with the induction
from ${\widehat{G}}$ to ${\widehat{\Z}}$ of the action
${\widehat{\af}} \colon G \to \Aut ( C^* (G, A, \af)),$
as described before Theorem~2.4 in~\cite{OPd0}.
Thus, we identify $C^* (\Z, A, \af_g)$ with the
set of all $f \in C \big( {\widehat{\Z}}, \, C^* (G, A, \af) \big)$
which are invatiant under the action
$\bt$ of~${\widehat{G}}$
defined by $\bt_{\ta} (f) (\zt) = {\widehat{\af}}_{\ta} (f ( \zt \ta ))$
for $\zt \in {\widehat{\Z}}$ and $\ta \in {\widehat{G}}.$
The dual action $\gm$ of ${\widehat{\Z}}$ on $C^* (\Z, A, \af_g)$
then becomes
$\gm_{\ld} (f) (\zt) = f (\ld^{-1} \zt)$
for $\ld, \zt \in {\widehat{\Z}}$
and $f \in C \big( {\widehat{\Z}}, \, C^* (G, A, \af) \big)^{\bt}.$

Clearly
$\Prim \big( C \big( {\widehat{\Z}}, \, C^* (G, A, \af) \big) \big)
  \cong {\widehat{\Z}} \times \Prim ( C^* (G, A, \af) ),$
and the action on it determined by~$\bt$ is
$\ta \cdot (\zt, P) = (\zt \ta^{-1}, \, \ta \cdot P)$
for $\ta \in {\widehat{G}} \S {\widehat{\Z}},$
$\zt \in {\widehat{\Z}},$
and $P \in \Prim ( C^* (G, A, \af) ).$
One easily checks that
the primitive ideal space of the fixed point algebra under $\bt$
is the quotient of ${\widehat{\Z}} \times \Prim ( C^* (G, A, \af) )$
by this action.
We write the image of $(\zt, P)$ as $[\zt, P].$
One further checks that the action on it determined by $\gm$ is
$\ld \cdot [\zt, P] = [\ld \zt, P]$
for $\ld, \zt \in {\widehat{\Z}}$
and $P \in \Prim ( C^* (G, A, \af) ).$

Let $\ta \in {\widehat{G}} \S {\widehat{\Z}}.$
Let $\zt \in {\widehat{\Z}}$
and let $P \in \Prim ( C^* (G, A, \af) ).$
We calculate, using
commutativity of ${\widehat{\Z}}$ at the first step,
the formula for the action of ${\widehat{G}}$
on ${\widehat{\Z}} \times \Prim ( C^* (G, A, \af) )$
at the second step,
and
triviality of the action
of ${\widehat{G}}$ on $\Prim ( C^* (G, A, \af) )$
at the third step,
\[
\ta \cdot [\zt, P]
= [\zt \ta, P]
= [\zt, \ta \cdot P]
= [\zt, P].
\]
So $\ta \in {\widetilde{\Gm}} (\af_g)$
by Lemma~3.4 of~\cite{Ks2}.
This completes the proof.
\end{proof}

The main part of the
proof that spectral freeness implies strong pointwise outerness
is contained in Lemma~\ref{L-608IdDecomp},
whose proof is based on Lemma 5.3.3 of~\cite{Ph1}.
We describe some notation
and give one preliminary lemma.
For a finite group~$G,$
let ${\mathcal{S}}_G$ be the set of all subsets $S \S G$
such that $1 \in S.$
Now let $A$ be a \ca,
let $\af \colon G \to \Aut (A)$ be an action of $G$ on~$A,$
and let $I \S A$ be an ideal.
For $S \in {\mathcal{S}}_G,$
we then define (following Lemma 5.3.3 of~\cite{Ph1})
$\af$-invariant ideals $I_S, I_S^{-} \subset A$ by
\[
I_S = \sum_{g \in G} \af_g \left( \bigcap_{h \in S} \af_h (I) \right)
\andeqn
I_S^{-} = \sum_{g \in G \setminus S} I_{S \cup \{ g \} }
        \S I_S.
\]
When $S = G,$
we take $I_S^{-} = \{ 0 \}.$

\begin{lem}\label{L_3724_Comp}
Let $\af \colon G \to \Aut (A)$ be an action of a finite group~$G$
on a \ca~$A.$
Let $I, L, M \S A$ be ideals
with $L \S M.$
Assume that
$(I_S \cap L) + I_S^{-} = (I_S \cap M) + I_S^{-}$
for all $S \in {\mathcal{S}}_G.$
Then $I_{ \{ 1 \} } \cap L = I_{ \{ 1 \} } \cap M.$
\end{lem}

The hypothesis says that $L$ and $M$
have the same image in $I_S / I_S^{-}$ for all $S \in {\mathcal{S}}_G.$

\begin{proof}[Proof of Lemma~\ref{L_3724_Comp}]
We prove that $I_S \cap L = I_S \cap M$
by downwards induction on $S \in {\mathcal{S}}_G.$
When $S = G,$
since $I_G^{-} = \{ 0 \},$
the hypothesis implies immediately that
$I_G \cap L = I_G \cap M.$

Now let $S \in {\mathcal{S}}_G$ with $S \neq G,$
and suppose that
$I_{S \cup \{ g \} } \cap L = I_{S \cup \{ g \} } \cap M$
for every $g \in G \setminus S.$
Then
\begin{align}\label{Eq:608ISM}
I_S^{-} \cap L
& = \left( \sum_{g \in G \setminus S} I_{S \cup \{ g \} } \right) \cap L
  = \sum_{g \in G \setminus S} (I_{S \cup \{ g \} } \cap L)
     \\
& = \sum_{g \in G \setminus S} (I_{S \cup \{ g \} } \cap M)
  = \left( \sum_{g \in G \setminus S} I_{S \cup \{ g \} } \right) \cap M
  = I_S^{-} \cap M.  \notag
\end{align}
Using $I_S^{-} \subset I_S$ at the first and fifth step,
the assumption
$(I_S \cap L) + I_S^{-} = (I_S \cap M) + I_S^{-}$
at the second step,
$L \subset M$ at the third step,
and~(\ref{Eq:608ISM}) at the fourth step,
we get
\begin{align*}
I_S \cap L
& = [(I_S \cap L) + I_S^{-}] \cap L
  = [(I_S \cap M) + I_S^{-}] \cap L
      \\
& = (I_S \cap M) + (I_S^{-} \cap L)
  = (I_S \cap M) + (I_S^{-} \cap M)
  = I_S \cap M.
\end{align*}
This completes the induction step.

The lemma follows by taking $S = \{ 1 \}.$
\end{proof}

\begin{lem}\label{L-608IdDecomp}
Let $\af \colon G \to \Aut (A)$ be an action of a finite group~$G$
on a \ca~$A.$
Let $g_0 \in G.$
Let $I, J \subset A$ be ideals such that
\[
I \subset J,
\,\,\,\,\,\
\af_{g_0} (I) = I,
\andeqn
\af_{g_0} (J) = J.
\]
Then there exist ideals $L, M \subset A$ such that
\[
L \subset M,
\,\,\,\,\,\
{\mbox{$\af_{g} (L) = L$ for all $g \in G,$}}
\andeqn
\af_{g_0} (M) = M,
\]
and such that there is an isomorphism $\ph$ from $M / L$ to an
$\af_{g_0}$-invariant subquotient of $J / I$
which intertwines the automorphisms of $M / L$ and $J / I$
induced by~$\af_{g_0}.$
\end{lem}

\begin{proof}
Let the notation be as before Lemma~\ref{L_3724_Comp}.
We first assume that
\[
(I_S \cap J) + I_S^{-} = (I_S \cap I) + I_S^{-}
\]
for all $S \in {\mathcal{S}}_G.$
Then $I_{ \{ 1 \} } \cap J = I_{ \{ 1 \} } \cap I$
by Lemma~\ref{L_3724_Comp}.
Since
$I \subset I_{ \{ 1 \} }$
(Lemma 5.3.3(3) of~\cite{Ph1}),
it follows that $I_{ \{ 1 \} } \cap J = I.$
Take $M = I_{ \{ 1 \} } + J$ and $L = I_{ \{ 1 \} }.$
The canonical isomorphism
$M / L \to J / (I_{ \{ 1 \} } \cap J) = J / I$
clearly intertwines the automorphisms induced by~$\af_{g_0}.$
So the lemma is proved in this case.

We may therefore assume that there is
$S \in {\mathcal{S}}_G$ such that
\[
(I_S \cap J) + I_S^{-} \neq (I_S \cap I) + I_S^{-}.
\]
Since $I_S^{-}$ is $\af$-invariant
(Lemma 5.3.3(1) of~\cite{Ph1}),
we simplify the notation by defining
$B = A / I_S^{-},$
$P = (I + I_S^{-}) / I_S^{-},$
and $Q = (J + I_S^{-}) / I_S^{-},$
and further letting $\bt \colon G \to \Aut ( B )$
be the action induced by~$\af.$
Then
\[
P_S^{-} = \{ 0 \},
\,\,\,\,\,\,
P_S \cap P
 \subset P_S \cap Q,
\andeqn
P_S \cap P
 \neq P_S \cap Q.
\]
The quotient
\[
Q / P
 = (J + I + I_S^{-}) / (I + I_S^{-})
 = J / [(I + I_S^{-}) \cap J]
\]
is an $\af_{g_0}$-invariant subquotient of $J / I,$
so it suffices to use $Q / P$
in place of $J / I$
in the statement to be proved.

By Lemma 5.3.3(5) of~\cite{Ph1},
there is a subgroup $H \subset G,$
an $H$-invariant ideal $N \subset P_S,$
a system $R$ of left coset representatives for $H$ in~$G,$
and a subset $R_0 \subset R,$
such that we have internal direct sum decompositions
\[
P_S = \bigoplus_{g \in R} \bt_g (N)
\andeqn
P \cap P_S
 = \bigoplus_{g \in R_0} \bt_g (N).
\]
If $g, h \in G,$
then $\bt_g (N) = \bt_h (N)$
if $g H = h H,$
and
$\bt_g (N) \cap \bt_h (N) = \E$
otherwise.
The relation
$\bt_{g_0} ( P ) = P$
therefore implies that
$g_0 R_0 H = R_0 H.$
So $g_0 (R \setminus R_0) H = (R \setminus R_0) H,$
from which it follows that
the ideal $E = \sum_{g \in R \setminus R_0} \bt_g (N)$
satisfies $\bt_{g_0} (E) = E.$
Also
$P_S
 = ( P \cap P_S ) \oplus E.$
Since $Q \cap P_S$
is an ideal in
$P_S = \bigoplus_{g \in R} \bt_g (N),$
we have
$Q \cap P_S
 = ( Q \cap P \cap P_S )
       \oplus ( Q \cap P ).$
Since $Q \cap P_S$
properly contains $P \cap P_S,$
we have $Q \cap E \neq \{ 0 \}.$
Therefore $Q \cap E$
is a nonzero $\bt_{g_0}$-invariant ideal
in~$B.$
It is $\bt_{g_0}$-equivariantly isomorphic to
$(Q \cap P_S)
    / (P \cap Q \cap P_S),$
which in turn is $\bt_{g_0}$-equivariantly isomorphic
to the subquotient
$[ ( Q \cap P_S ) + P ] / P$
of $Q / P.$
This completes the proof.
\end{proof}

\begin{thm}\label{T_3725_FgSpF}
Let $\af \colon G \to \Aut (A)$
be an action
of a finite group~$G$ on a \ca~$A.$
Then \tfae:
\begin{enumerate}
\item\label{T_3725_FgSpF_SpFr}
$\af$ is spectrally free.
\item\label{T_3725_FgSpF_PtOut}
$\af$ is strongly pointwise outer.
\item\label{T_3725_FgSpF_GmTl}
${\widetilde{\Gm}} (\af_g) \neq \{ 1 \}$
for every $g \in G \SM \{ 1 \}.$
\end{enumerate}
\end{thm}

\begin{proof}
That (\ref{T_3725_FgSpF_GmTl}) implies~(\ref{T_3725_FgSpF_SpFr})
is Proposition~\ref{P_3723_StCSpToSpFr}
(even when $G$ is not finite).
That (\ref{T_3725_FgSpF_PtOut}) implies~(\ref{T_3725_FgSpF_GmTl})
is Proposition~\ref{P_3723_SPOtoStCsp}.

We prove that (\ref{T_3725_FgSpF_SpFr})
implies~(\ref{T_3725_FgSpF_PtOut}).
Suppose that $\af$ is not strongly pointwise outer.
Then there exist $g \in G \setminus \{ 1 \}$
and $\af_g$-invariant ideals $I, J \subset A$
such that $I \subset J,$ $I \neq J,$
and the automorphism $\bt$ of $J / I$ induced by $\af_g$ is inner.
By Lemma~\ref{L-608IdDecomp},
we may assume that $I$ is $\af$-invariant.
Now $\Gm_{\mathrm{B}} (\bt) = \{ 1 \}.$
Therefore the automorphism of $A / I$
induced by $\af_g$ is spectrally trivial,
so $\af$ is not spectrally free.
\end{proof}

\begin{qst}\label{Q-FreeAct403xx-4}
Does strong pointwise outerness imply spectral freeness
for groups which are not finite?
\end{qst}

We give a result on spectral freeness for actions
on tensor products (Proposition~\ref{P-610SfrTP} below).
One hopes to get the conclusion under much weaker hypotheses:
simplicity of $B$ should not be necessary,
and one hopes that the trivial action on~$B$
can be replaced by an arbitrary action.
We have not investigated such results.
The statement we give is easy to prove,
and is sufficient for a later application
(in the proof of Corollary~\ref{C-SpFrTopZimZ}).

We start with a lemma.

\begin{lem}\label{L-610GmBTP}
Let $\af \colon G \to \Aut (A)$
be an action
of a locally compact abelian group $G$ on a \ca~$A.$
Let $B$ be any simple nuclear \ca,
and let $\gm \colon G \to \Aut (B \otimes A)$
be the action $\gm_g = \id_{B} \otimes \af_g$
for $g \in G.$
Then $\Gm_{\mathrm{B}} (\gm) = \Gm_{\mathrm{B}} (\af).$
\end{lem}

\begin{proof}
We use the criterion of Proposition~2.1 of~\cite{Ks7}.
We write ${\widehat{G}}$ multiplicatively rather than additively,
and for $\Om \subset {\widehat{G}}$ and $k \in \N,$
we let
\[
\Om^k = \big\{ \ta_1 \ta_2 \cdots \ta_k \colon
      \ta_1, \ta_2, \ldots, \ta_k \in \Om \big\}.
\]
We further denote the action of ${\widehat{G}}$ on
$\Prim (C^* (G, A, \af))$ coming from ${\widehat{\af}}$
by $(\ta, x) \mapsto \ta x.$

Translated from ideals in $C^* (G, A, \af)$
to open subsets of $\Prim ( C^* (G, A, \af) ),$
Proposition~2.1 of~\cite{Ks7} states that
an element $\ta \in {\widehat{G}}$
is in $\Gm_{\mathrm{B}} (\af)$ \ifo\  %
for every open set $\Om \subset {\widehat{G}}$ containing~$\ta,$
every $n \in \N,$
and every open set $U \subset \Prim ( C^* (G, A, \af))$
such that ${\widehat{G}} U$ is dense
in $\Prim ( C^* (G, A, \af) ),$
there exist
$\sm_1 \in \Om,$ $\sm_2 \in \Om^2,$ \ldots, $\sm_n \in \Om^n$
such that
$U \cap \sm_1 U \cap \sm_2 U \cap \cdots \cap \sm_n U \neq \E.$
Since $B$ is simple and nuclear, and since
$C^* ( G, \, B \otimes A, \, \gm ) \cong B \otimes C^* (G, A, \af),$
there is an isomorphism
$\Prim ( C^* ( G, \, B \otimes A, \, \gm ) )
  \cong \Prim ( C^* (G, A, \af) )$
which is equivariant for the dual actions.
Therefore the criterion for $\ta$ to be in $\Gm_{\mathrm{B}} (\af)$
holds
\ifo\  the corresponding criterion
for $\ta$ to be in $\Gm_{\mathrm{B}} (\gm)$ holds.
\end{proof}

\begin{prp}\label{P-610SfrTP}
Let $\af \colon G \to \Aut (A)$
be an action
of a discrete group $G$ on a \ca~$A.$
Let $B$ be any simple nuclear \ca,
and let $\gm \colon G \to \Aut (B \otimes A)$
be the action $\gm_g = \id_{B} \otimes \af_g$
for $g \in G.$
\begin{enumerate}
\item\label{P-610SfrTP-SNT}
If $\af$ is pointwise spectrally nontrivial,
then so is~$\gm.$
\item\label{P-610SfrTP-SFr}
If $\af$ is spectrally free,
then so is~$\gm.$
\end{enumerate}
\end{prp}

\begin{proof}
We prove~(\ref{P-610SfrTP-SNT}).
Let $\gm \in G \setminus \{ 1 \}.$
Since $B$ is simple and nuclear,
the map $J \mapsto B \otimes J$
defines a bijection from the $\af_g$-invariant ideals of~$A$
to the $\gm_g$-invariant ideals of $B \otimes A.$
Moreover,
Lemma~\ref{L-610GmBTP} implies that for each such~$J,$
we have $\Gm_{\mathrm{B}} (\gm_g |_J) = \Gm_{\mathrm{B}} (\af_g |_J).$
The desired conclusion follows.

Part~(\ref{P-610SfrTP-SFr}) follows from part~(\ref{P-610SfrTP-SNT})
because $J \mapsto B \otimes J$
defines a bijection from the $\af$-invariant ideals of~$A$
to the $\gm$-invariant ideals of $B \otimes A,$
and because $(B \otimes A) / (B \otimes J) = B \otimes A / J.$
\end{proof}

\section{Rokhlin actions of~$\Z$}\label{Sec:ZRP}

Let $A$ be a unital \ca,
and let $\af \colon \Z \to \Aut (A)$ be an action
with the Rokhlin property.
In this section, we show directly that
all ideals in $C^* (\Z, A, \af)$
are crossed product ideals.
In particular, $A$ separates the ideals in $C^* (\Z, A, \af).$
The main step is a kind of averaging result
which seems of independent interest
(Lemma~\ref{L-RPcsq323}).
This is a somewhat old result;
Lemma~\ref{L-RPcsq323} and Theorem~\ref{IdealZRP}
are attributed to us in~\cite{Ln}.
(See Lemma~4.1 of~\cite{Ln}.)

We use this result to prove that such crossed products
preserve the ideal and \pj{} properties.
The Rokhlin property is much stronger than needed
for this.
Some condition is needed, though,
since crossed products by the trivial action of~$\Z$
never have either the ideal or projection properties.
We also use this result to show that
actions of $\Z$ with the Rokhlin property
satisfy the hypotheses of some of our later theorems.

Recall (Definition~2.5 of the survey article~\cite{Iz1})
that an automorphism $\af$ of a unital \ca~$A$
(strictly speaking, the action of $\Z$ generated by~$\af$)
is said to have the Rokhlin property if
for every $\ep > 0,$ every $n \in \N,$
and every finite set $F \subset A,$
there exist orthogonal \pj s
\[
p_1, p_2, \ldots, p_n, q_1, q_2, \ldots, q_{n + 1} \in A
\]
such that $\sum_{j = 1}^n p_j + \sum_{j = 1}^{n + 1} q_j = 1,$
such that
$\| \af ( p_j ) - p_{j + 1} \| < \ep$ for $j = 1, 2, \ldots, n - 1$
and $\| \af ( q_j ) - q_{j + 1} \| < \ep$ for $j = 1, 2, \ldots, n,$
and such that $\| p_j a - a p_j \| < \ep$
for $j = 1, 2, \ldots, n$ and $a \in F$
and $\| q_j a - a q_j \| < \ep$
for $j = 1, 2, \ldots, n + 1$ and $a \in F.$

\begin{lem}\label{L-RPcsq323}
Let $A$ be a unital \ca,
and let $\af \in \Aut (A)$ have the Rokhlin property.
Regard $A$ as a subalgebra of $C^* (\Z, A, \af)$ in the usual way,
and let $E \colon C^* (\Z, A, \af) \to A$
be the standard conditional expectation.
Then for every finite set $F \subset C^* (\Z, A, \af)$
and every $\ep > 0,$
there exist $m \in \N$ and \mops\  %
$e_1, e_2, \ldots, e_m \in A$
such that $\sum_{j = 1}^m e_j = 1$ and
\[
\left\| E (a) - \sum_{j = 1}^m e_j a e_j \right\| < \ep
\]
for all $a \in F.$
\end{lem}

\begin{proof}
Recall that $u$ denotes the standard unitary in~$C^* (\Z, A, \af)$
which implements~$\af.$
Choose a finite set $S \subset A$ and $N \in \Z$ such that for
every $b \in F$ there exist
\[
b_{- N}, b_{ - N + 1}, \ldots, b_{N - 1}, b_N \in S
\]
such that
\[
\left\| b - \sum_{k = - N}^N b_k u^k \right\| < \frac{\ep}{3}.
\]
Set
\[
M = \sup_{a \in S} \| a \|,
\,\,\,\,\,\,
n = 2 N + 1,
\andeqn
\dt = \frac{\ep}{3 n (2 n + 1) (6 n M + 1)}.
\]
Apply the Rokhlin property
with $S$ in place of $F$ and $\dt$ in place of~$\ep,$
and let
\[
p_1, p_2, \ldots, p_n, q_1, q_2, \ldots, q_{n + 1} \in A
\]
be the resulting \pj s.

We claim that for
$k = - N, \, - N + 1, \, \ldots, \, N - 1, \, N,$
except for $k = 0,$
we have $\| p_j \af^k (p_j) \| < 6 n \dt$
for $j = 1, 2, \ldots, n$
and $\| q_j \af^k (q_j) \| < 6 n \dt$
for $j = 1, 2, \ldots, n + 1.$
We first observe that an induction argument gives
$\| \af^l (p_j) - p_{j + l} \| < l \dt$
whenever $1 \leq j \leq n$ and $0 \leq l \leq n - j,$
and $\| \af^l (q_j) - q_{j + l} \| < l \dt$
whenever $1 \leq j \leq n + 1$ and $0 \leq l \leq n + 1 - j.$
In particular,
under the given conditions,
we always have
\[
\| \af^l (p_j) - p_{j + l} \| < n \dt
\andeqn
\| \af^l (q_j) - q_{j + l} \| < n \dt.
\]
Next,
\begin{align*}
\big\| \af (p_n + q_{n + 1}) - (p_1 + q_1) \big\|
& = \big\| \af (1 - p_n - q_{n + 1}) - (1 - p_1 - q_1) \big\|
     \\
& = \left\| \af \left( \sum_{j = 1}^{n - 1} p_j + \sum_{j = 1}^n q_j
                             \right)
            - \left( \sum_{j = 1}^{n - 1} p_{j + 1}
               + \sum_{j = 1}^n q_{j + 1} \right) \right\|
     \\
& < (2 n - 1) \dt.
\end{align*}

Now consider $\| p_j \af^k (p_j) \|$ when
$k \in \{ 1, \, \ldots, \, N - 1, \, N \}$
and $j + k \leq n.$
Since $p_j p_{j + k} = 0$
and $\| \af^k (p_j) - p_{j + k} \| < n \dt,$
we get $\| p_j \af^k (p_j) \| < n \dt \leq 6 n \dt.$
Similarly,
if $j + k \leq n + 1,$
then $\| q_j \af^k (q_j) \| < n \dt \leq 6 n \dt.$
Next, suppose $j + k > n.$
Set $r = j + k - n.$
We have
\begin{align*}
\big\| \af^k (p_j + q_{j + 1}) - (p_r + q_r) \big\|
& \leq \big\| \af^{n - j} (p_j + q_{j + 1}) - (p_n + q_{n + 1}) \big\|
     \\
& \hspace*{4em} {\mbox{}}
       + \big\| \af (p_n + q_{n + 1}) - (p_1 + q_1) \big\|
     \\
& \hspace*{4em} {\mbox{}}
       + \big\| \af^{r - 1} (p_1 + q_1) - (p_r + q_r) \big\|
     \\
& < 2 n \dt + (2 n - 1) \dt + 2 n \dt
  < 6 n \dt.
\end{align*}
Since $\af^k (p_j) \leq \af^k (p_j + q_{j + 1})$
and $p_j (p_r + q_r) = 0$
(because $k \leq N$ and $n > N$ imply $r \neq j$),
we get $\| p_j \af^k (p_j) \| < 6 n \dt.$
A similar argument using $r = j + k - n - 1$
shows that if $j + k > n + 1,$
then $\| q_j \af^k (q_j) \| < 6 n \dt.$

Finally, suppose
$k \in \{ - N, \, - N + 1, \, \ldots, \, - 1 \}.$
Then
\[
\| p_j \af^k (p_j) \|
  = \| \af^k (p_j) p_j \|
  = \| p_j \af^{- k} (p_j) \|
  < 6 n \dt,
\]
and similarly $\| q_j \af^k (q_j) \| < 6 n \dt.$
This proves the claim.

Now let
$c = \sum_{k = - N}^N c_k u^k$
with
$c_{- N}, c_{ - N + 1}, \ldots, c_{N - 1}, c_N \in S.$
We claim that
\[
\left\| \sum_{j = 1}^n p_j c p_j + \sum_{j = 1}^{n + 1} q_j c q_j
   - E (c) \right\|
 < \frac{\ep}{3}.
\]

Set
\[
R = \{ - N, \, - N + 1, \, \ldots, \, N - 1, \, N \}
    \setminus \{ 0 \}.
\]
First, let $a \in S$ and
$k \in R.$
Then for $j = 1, 2, \ldots, n$ we have
\begin{align*}
\| p_j a u^k p_j \|
& = \| p_j a \af^k ( p_j ) u^k \|
     \\
& \leq \| p_j a - a p_j \|
        + \| a \| \cdot \| p_j \af^k (p_j) \| \cdot \| u^k \|
  < \dt + 6 n M \dt
  = (6 n M + 1) \dt.
\end{align*}
Also, for $k = 0,$
we get
\[
\| p_j a p_j - p_j a \|
  \leq \| p_j \| \cdot \| a p_j - p_j a \|
  < \dt.
\]
Similarly, for $j = 1, 2, \ldots, n + 1$ we have
\[
\| q_j a u^k q_j \| < (6 n M + 1) \dt
\andeqn
\| q_j a q_j - q_j a \| < \dt.
\]

For $c$ as above,
we now have
\begin{align*}
\left\| \sum_{j = 1}^n p_j c p_j + \sum_{j = 1}^{n + 1} q_j c q_j
   - E (c) \right\|
& = \left\| \sum_{j = 1}^n p_j c p_j + \sum_{j = 1}^{n + 1} q_j c q_j
    - \sum_{j = 1}^n p_j c_0 - \sum_{j = 1}^{n + 1} q_j c_0 \right\|
    \\
& \leq \sum_{k \in R} \left( \sum_{j = 1}^n \| p_j c_k u^k p_j \|
                           + \sum_{j = 1}^{n + 1} \| q_j c_k u^k q_j \|
                                    \right)
     \\
& \hspace*{2em} {\mbox{}}
          + \sum_{j = 1}^n \| p_j c_0 p_j - p_j c_0 \|
                + \sum_{j = 1}^{n + 1} \| q_j c_0 q_j - q_j c_0 \|
     \\
& < (2 N) (2 n + 1) (6 n M + 1) \dt + (2 n + 1) \dt
     \\
& \leq (2 N + 1) (2 n + 1) (6 n M + 1) \dt
     \\
& = n (2 n + 1) (6 n M + 1) \dt
  \leq \frac{\ep}{3}.
\end{align*}
This proves the claim.

To prove the lemma,
with $e_1, e_2, \ldots, e_m$ being
\[
p_1, p_2, \ldots, p_n, q_1, q_2, \ldots, q_{n + 1},
\]
let $a \in F.$
By hypothesis, there is $c$ of the form above
such that $\| c - a \| < \tfrac{1}{3} \ep.$
Then
\[
\left\| E (a) - \sum_{j = 1}^m e_j a e_j \right\|
  \leq \| E (a) - E (c) \|
    + \left\| E (c) - \sum_{j = 1}^m e_j c e_j \right\|
    + \| c - a \|
  < \frac{\ep}{3} + \frac{\ep}{3} + \frac{\ep}{3}
  = \ep.
\]
This completes the proof.
\end{proof}

As a corollary,
we obtain the following result.

\begin{thm}\label{IdealZRP}
Let $A$ be a unital \ca,
and let $\af \in \Aut (A)$ have the Rokhlin property.
Then for every $n \neq 0,$
the algebra $A$ separates the ideals in $C^* (\Z, A, \af^n).$
\end{thm}

The proof uses the following standard lemma.
We recall that if
$\af \colon G \to \Aut (A)$ is an action
of a group~$G$ on a \ca~$A,$
and if $I \S A$ is a $G$-invariant ideal,
then $C^* (G, I, \af) \to C^* (G, A, \af)$ is injective
(Lemma~2.8.2 of~\cite{Ph1})
and $C^*_{\mathrm{r}} (G, I, \af) \to C^*_{\mathrm{r}} (G, A, \af)$
is injective
(Proposition 7.7.9 of~\cite{Pd1}).

\begin{lem}\label{P-JIntA323xx}
Let $A$ be a \ca, let $G$ be a discrete group, and let
$\af \colon G \to \Aut (A)$ be an action of $G$ on~$A.$
Regard $A$ as a subalgebra of $C^* (G, A, \af)$
in the usual way.
Let $J \subset C^* (G, A, \af)$ be an ideal.
Then $I = J \cap A$ is a $\af$-invariant ideal in~$A,$
and $C^* (G, I, \af) \subset J.$
The analogous statement holds for ideals
in $C^*_{\mathrm{r}} (G, A, \af).$
\end{lem}

\begin{proof}
We prove the statement involving $C^* (G, A, \af)$;
the proof for the statement involving $C^*_{\mathrm{r}} (G, A, \af)$
is the same.

For $\af$-invariance,
let $a \in I$ and $g \in G.$
Then $\af_g (a) = u_g a u_g^*$ is in both $J$ and~$A.$

We have $C^* ( G, I, \af ) \subset J$
because, if $a \in I$ and $g \in G,$
then $a \in J,$ so $a u_g \in J.$
\end{proof}

\begin{proof}[Proof of Theorem~\ref{IdealZRP}]
Let $J \subset C^* (\Z, A, \af^n)$ be an ideal.
Set $I = J \cap A,$
which, by Lemma~\ref{P-JIntA323xx},
is an $\af^n$-invariant ideal in~$A$
such that $C^* ( \Z, I, \af^n ) \subset J.$

Let $E \colon C^* (\Z, A, \af^n) \to A$
be the standard conditional expectation.
We claim that $E (J) \subset I.$
So let $a \in J,$ and let $\ep > 0.$
Think of $a$ as an element of
$C^* (\Z, A, \af) \supset C^* (\Z, A, \af^n),$
and note that $E$ is the restriction to $C^* (\Z, A, \af^n)$
of the standard conditional expectation from $C^* (\Z, A, \af)$ to~$A.$
Lemma~\ref{L-RPcsq323} provides $m \in \N$ and elements
$e_1, e_2, \ldots, e_m \in A$
such that
\[
\left\| E (a) - \sum_{j = 1}^m e_j a e_j \right\| < \ep.
\]
Since
$\sum_{j = 1}^m e_j a e_j \in J$ and $\ep > 0$ is arbitrary,
it follows that $E (a) \in J.$
The claim follows.

We finish the proof by showing that $J \subset C^* ( \Z, I, \af^n ).$
Let $a \in J.$
Let $u \in C^* (\Z, A, \af^n)$ be the standard unitary,
that is, the one implementing the automorphism~$\af^n.$
For $k \in \Z,$
define $a_k = E (a u^{-k}) \in I.$
Then for every $m \in \Z,$ the Ces\`{a}ro sum
\[
b_m
 = \sum_{k = - m}^{m} \left( 1 - \frac{ |k| }{ m + 1} \right) a_k u^k
\]
is in $C^* ( \Z, I, \af^n ),$
and Theorem VIII.2.2 of~\cite{Dv} implies that
$\limi{m} b_m = a.$
So $a \in C^* ( \Z, I, \af^n ).$
\end{proof}

Recall (Definition~1 of~\cite{Psn})
that a \ca~$A$ has the projection property
if every ideal in~$A$ has an approximate identity consisting
of projections.

\begin{cor}\label{C:C_0005}
Let $A$ be a unital \ca,
and let $\af \in \Aut (A)$ have the Rokhlin property.
\begin{enumerate}
\item\label{C:C_0005-IdP}
If $A$ has the ideal property,
then $C^* (\Z, A, \alpha)$ has the ideal property.
\item\label{C:C_0005-PjP}
If $A$ has the projection property,
then $C^* (\Z, A, \alpha)$ has the projection property.
\end{enumerate}
\end{cor}

\begin{proof}
The first part follows from Theorem~\ref{IdealZRP}
and the fact that if $B$ is any C*-algebra
which is generated as an ideal by its \pj s,
and $\bt \in \Aut (B)$ is arbitrary,
then $C^* (\Z, B, \bt)$ is generated as an ideal by its \pj s.
The proof of the second part is similar.
\end{proof}

Finally,
we prove that, for actions of~$\Z,$
the Rokhlin property implies spectral freeness.
The same is true for actions of finite groups,
but we omit the proof.

\begin{prp}\label{P-611-RPInpSFr}
Let $A$ be a unital \ca,
and let $\af \in \Aut (A)$ have the Rokhlin property.
Then $\af$ generates an exact spectrally free action of $\Z$ on~$A.$
\end{prp}

\begin{proof}
It is clear that the Rokhlin property passes to
the induced automorphism on a
quotient by an $\af$-invariant ideal.
So we need only prove that
the action is pointwise spectrally nontrivial.

Fix $n \in \Z \setminus \{ 0 \}$
and let $I \subset A$ be an $\af^n$-invariant ideal.
Let $\bt \colon S^1 \to \Aut ( C^* ( \Z, A, \af^n ) )$
be the dual action.
Theorem~\ref{IdealZRP} implies that all ideals
in $C^* ( \Z, A, \af^n )$
are $\bt$-invariant.
Since $C^* ( \Z, I, \af^n )$
is an ideal in $C^* ( \Z, A, \af^n ),$
it follows that all ideals in
$C^* ( \Z, I, \af^n )$
are invariant under the dual of the action of $\af^n$ on~$I.$
It now follows from Lemma~3.4 of~\cite{Ks1}
that ${\widetilde{\Gm}} (\af^{n} |_I) = S^1.$
In particular, ${\widetilde{\Gm}} (\af^{n} |_I) \neq \{ 1 \},$
so the larger set $\Gm_{\mathrm{B}} (\af^{n} |_I)$ is nontrivial.
\end{proof}

\section{Using spectral freeness}\label{Sec:UsingSpFree}

\indent
In this section,
we give the technical results
(Lemma~\ref{L-L8_012Mod} and Corollary~\ref{C-HerIso606})
which we use to prove that crossed products by spectrally free
actions preserve some interesting properties of \ca{s}.

The proof of Lemma~\ref{L-L8_012Mod} is a modification of part
of the proof of Lemma~10 of~\cite{KsKu}.
(Also see~\cite{Ks2}
and Lemma~3.2 of~\cite{RrSr}.)
As preparation, we formally state the improvement of
Lemma~3.2 of~\cite{Ks2}
mentioned in Remark~2.3 of~\cite{Ks7}.
It is essentially Lemma~7.1 of~\cite{OPd3},
but without the separability condition in~\cite{OPd3}.

\begin{lem}[Remark~2.3 of~\cite{Ks7}]\label{L-Ksh608}
Let $A$ be a \ca,
let $a \in A_{+},$
let $n \in \N,$
let $a_1, a_2, \ldots, a_n \in A,$
and let $\af_1, \af_2, \ldots, \af_n \in \Aut (A)$
be spectrally nontrivial.
Then for every $\ep > 0$ there is $x \in A_{+}$ with $\| x \| = 1$
such that $\| x a x \| > \| a \| - \ep$
and $\| x a_k \af_k (x) \| < \ep$ for $k = 1, 2, \ldots, n.$
\end{lem}

After we proved the following lemma,
the papers \cite{HO} and~\cite{GrSr} were posted on the arXiv.
Lemma~5.1 of~\cite{HO} is a weaker version of our lemma.
(It assumes simplicity of the algebra.)
Lemma~3.11 of~\cite{GrSr} is more general
(applying to partial actions),
but is stated differently.
The proofs of both are similar to our proof.

\begin{lem}\label{L-L8_012Mod}
Let $\af \colon G \to \Aut (A)$
be a pointwise spectrally nontrivial action
of a discrete group $G$ on a \ca~$A.$
Identify $A$ with a subalgebra of $C^*_{\mathrm{r}} (G, A, \af)$
in the usual way.
Let $a \in C^*_{\mathrm{r}} (G, A, \af)_{+} \setminus \{ 0 \}.$
Then there is $z \in C^*_{\mathrm{r}} (G, A, \af)$
such that $z a z^*$ is a nonzero element of~$A.$
\end{lem}

\begin{proof}
Let $E \colon C^*_{\mathrm{r}} (G, A, \af) \to A$
and $u_g$ for $g \in G$
be the standard conditional expectation and unitaries,
as in the introduction.
For $g \in G,$ define $a_g = E (a u_g^*).$
Since $E$ is faithful,
$a_1 \neq 0.$
By scaling,
we may therefore assume that $\| a_1 \| = 1.$
Then $\| a \| \geq 1.$

Set $\ep = \frac{1}{7}.$

We now follow the third paragraph
of the proof of Lemma~10 of~\cite{KsKu},
with, for convenience, an extra normalization.
Set
\[
\dt
 = \min \left( \| a \|^{1/2}, \, \frac{1}{6 \| a \|^{1/2}}, \,
          \frac{\ep}{27 \| a \|^{3/2} } \right).
\]
Choose a finite set $S_0 \subset G$
and an element $b \in C^*_{\mathrm{r}} (G, A, \af)$
of the form $b = \sum_{g \in S_0} b_g u_g$
with $b_g \in A$ for $g \in S_0,$
such that $\| b - a^{1/2} \| < \dt.$
Then
\begin{align*}
\| b^* b - a \|
& \leq \| b^* \| \cdot \| b - a^{1/2} \|
           + \| b^* - a^{1/2} \| \cdot \| a^{1/2} \|
\\
& < ( \| a \|^{1/2} + \dt) \dt + \dt \| a \|^{1/2}
  \leq 3 \| a \|^{1/2} \dt.
\end{align*}
Therefore $\| E (b^* b) - a_1 \| < 3 \| a \|^{1/2} \dt.$
Since $3 \| a \|^{1/2} \dt < 1,$
we have $E (b^* b) \neq 0.$
Set
\[
c = \| E (b^* b) \|^{-1} b^* b.
\]
Then there exist a finite set $S \subset G \setminus \{ 1 \}$
and elements $c_1 \in A$ and $c_g \in A$ for $g \in S$
such that $c = c_1 + \sum_{g \in S} c_g u_g.$
Moreover, $c_1 = \| E (b^* b) \|^{-1} E (b^* b),$
so $\| c_1 \| = 1.$
We also have, using $\| b^* b \| < \| a \| + 3 \| a \|^{1/2} \dt,$
\begin{align*}
\| c - a \|
& \leq \| c - b^* b \| + \| b^* b - a \|
  < \left( \frac{1}{1 - 3 \| a \|^{1/2} \dt} - 1 \right) \| b^* b \|
            + 3 \| a \|^{1/2} \dt
        \\
& \leq \left( \frac{3 \| a \|^{1/2} \dt}{1 - 3 \| a \|^{1/2} \dt}\right)
                      \left( \| a \| + 3 \| a \|^{1/2} \dt \right)
            + 3 \| a \|^{1/2} \dt
        \\
& < 6 \| a \|^{1/2} \dt (4 \| a \| ) + 3 \| a \|^{1/2} \dt
  \leq 27 \| a \|^{3/2} \dt
  \leq \ep.
\end{align*}


Lemma~\ref{L-Ksh608} now provides $x \in A_{+}$ such that
\[
\| x \| = 1,
\,\,\,\,\,\,
\| x c_1 x \| > 1 - \ep,
\andeqn
\max_{g \in S} \| x c_g \af_g (x) \| < \frac{\ep}{\card (S)}.
\]
%
%
We estimate
\begin{align}\label{Eq:612St}
\| x a x - x a_1 x \|
& \leq 2 \| a - c \| + \| x c x - x c_1 x \|
  \leq 2 \| a - c \| + \sum_{g \in S} \| x c_g u_g x \|
\\
& = 2 \| a - c \| + \sum_{g \in S} \| x c_g \af_g (x) u_g \|
  < 2 \ep + \ep
  = 3 \ep
  \notag
\end{align}
and
\begin{equation}\label{Eq:612StSt}
\| x a_1 x \|
  \geq \| x c_1 x \| - \| a - c \|
  > 1 - \ep - \ep
  = 1 - 2 \ep.
\end{equation}

Recall that for $y \in A_{+}$ and $\rh \geq 0,$
the expression $(y - \rh)_{+}$ is defined to be the
result of applying functional calculus to~$y$
with the function $t \mapsto \max (0, \, t - \rh).$
Using the Lemma~2.2 of~\cite{KR2} and~(\ref{Eq:612St}),
we find $d \in C^*_{\mathrm{r}} (G, A, \af)$
such that
\[
d x a x d^* = (x a_1 x - 4 \ep)_{+}.
\]
Since $\ep = \frac{1}{7}$ and $\| a_1 \| = 1,$
the estimate~(\ref{Eq:612StSt})
implies that $(x a_1 x - 4 \ep)_{+} \neq 0.$
Clearly $(x a_1 x - 4 \ep)_{+} \in A,$
so we complete the proof by taking $z = d x.$
\end{proof}

The following definition from~\cite{Sr} provides convenient
terminology for an important corollary.

\begin{dfn}[Definition~1.9 of~\cite{Sr}]\label{D_2Y30IntP}
Let $\af \colon G \to \Aut (A)$
be an action
of a discrete group $G$ on a \ca~$A.$
We say that $\af$ has
the {\emph{intersection property}}
if for every nonzero ideal $I \S C^*_{\mathrm{r}} (G, A, \af),$
we have $I \cap A \neq 0.$
We further say that $\af$ has
the {\emph{residual intersection property}}
if for every $\af$-invariant ideal $J \S A,$
the induced action of $G$ on $A / J$ has the intersection property.
\end{dfn}

In the separable case,
since spectral nontriviality implies proper outerness
(Remark 2.5 of~\cite{Ks7}),
Corollary~\ref{C_2Y30SpNtrIntP} (implicitly)
and Proposition~\ref{P-SNAQ-Sep608} below
are contained in Remark~2.23 of~\cite{Sr}.
Our proof uses less machinery.
(There is a misprint in~\cite{Sr}:
the reference to Theorem~1.10 of~\cite{Sr} there
should be to Theorem~1.13 of~\cite{Sr}.)
These results were also independently obtained,
in close to the form in which we give them
but generalized to partial actions,
as Theorem~3.12 and Corollary~3.13 of~\cite{GrSr}.

\begin{cor}\label{C_2Y30SpNtrIntP}
Let $\af \colon G \to \Aut (A)$
be a pointwise spectrally nontrivial action
of a discrete group $G$ on a \ca~$A.$
Then $\alpha$ has the intersection property.
\end{cor}

\begin{proof}
Let $I \S C^*_{\mathrm{r}} (G, A, \af)$ be a
nonzero ideal.
Let $a \in I$ be a nonzero positive element.
Lemma~\ref{L-L8_012Mod}
provides $z \in C^*_{\mathrm{r}} (G, A, \af)$
such that $z a z^*$ is a nonzero element of~$A.$
Since $z a z^* \in I,$
we have $I \cap A \neq 0.$
\end{proof}

\begin{prp}\label{P-SNAQ-Sep608}
Let $\af \colon G \to \Aut (A)$
be a spectrally free action
of a discrete group $G$ on a \ca~$A.$
Suppose that $\af$ is exact
(in the sense of Definition~1.5 of~\cite{Sr}).
Then $A$ separates the ideals
of $C^*_{\mathrm{r}} (G, A, \af)$
(in the sense of Definition~\ref{D_3718_SepId}).
\end{prp}

\begin{proof}
By Theorem~1.13 of~\cite{Sr},
it is enough to prove that $\af$ has the residual intersection property.
{}From the definitions,
it is enough to prove that if $\af$ is spectrally nontrivial,
then $\af$ has the intersection property.
This is Corollary~\ref{C_2Y30SpNtrIntP}.
\end{proof}

Any converse to Proposition~\ref{P-SNAQ-Sep608} must have
very restrictive hypotheses.
According to Theorem~1.13 of~\cite{Sr},
the algebra $A$ separates the ideals
of $C^*_{\mathrm{r}} (G, A, \af)$ \ifo\  the action is exact
and has the residual intersection property.
The trivial action of a nonabelian free group on~$\C$
gives a simple reduced crossed product.
Example 4.2.3 of~\cite{Ph1}
contains an action $\af \colon \Z_2 \times \Z_2 \to M_2$
such that $C^* (\Z_2 \times \Z_2, \, M_2, \, \af)$
is simple but $\af_g$ is inner for every $g \in G.$
In both cases,
the action is exact and has the residual intersection property,
but is very far from being spectrally free.

\begin{cor}\label{C_2Z01_2nd}
Let $\af \colon G \to \Aut (A)$
be a pointwise spectrally nontrivial action
of a discrete group $G$ on a \ca~$A.$
For any
nonzero ideal $I \S C^*_{\mathrm{r}} (G, A, \af)$
there exists a nonzero ideal $J \S A$
such that $J \S I.$
\end{cor}

\begin{proof}
Take $J = I \cap A$ in Corollary~\ref{C_2Y30SpNtrIntP}.
\end{proof}

\begin{cor}\label{C_2Z01_3rd}
Let $\af \colon G \to \Aut (A)$
be a pointwise spectrally nontrivial action
of a discrete group $G$ on a \ca~$A.$
Suppose that every nonzero ideal in $A$ contains a \nzp.
Then every nonzero ideal in $C^*_{\mathrm{r}} (G, A, \af)$
contains a \nzp.
\end{cor}

The main part of the following lemma,
namely $B \cong C,$
is in~1.4 of~\cite{Cu0},
with a slightly different proof.
Primarily for use in Section~\ref{Sec:Start},
we need additional information.

\begin{proof}
This is immediate from Corollary~\ref{C_2Z01_2nd}.
\end{proof}

\begin{lem}\label{L-Polar606}
Let $A$ be a \ca, and let $a \in A.$
In $A^{**},$
let $a = v (a^* a)^{1/2}$ be the polar decomposition of~$a.$
Set $B = {\overline{ (a^* a) A (a^* a)}}$
and $C = {\overline{ (a a^*) A (a a^*)}}.$
Let $x \in A.$
Then:
\begin{enumerate}
\item\label{L-Polar606_2Z10One}
$x^* x \in B$ implies $x v^* \in A$
and $x x^* \in B$ implies $v x \in A.$
\item\label{L-Polar606_2Z10Two}
$x \in B$ implies $v x v^* \in C.$
\item\label{L-Polar606_2Z10Three}
$x^* x \in B$ implies $x v^* v = x$
and $x x^* \in B$ implies $v^* v x = x.$
\end{enumerate}
Moreover, the formula $\ph (x) = v x v^*$ defines an isomorphism
$\ph \colon B \to C$
such that $\ph (a^* a) = a a^*$
and such that for every
$x \in B_{+},$
we have
$x \sim \ph (x)$ in~$A.$
\end{lem}

\begin{proof}
It follows from Proposition~1.3 of~\cite{Cu0}
that for every \cfn{} $f \colon [0, \I) \to [0, \I)$
such that $f (0) = 0,$
we have $v f (a^* a) \in A.$
Therefore also $f (a^* a) v^* \in A.$

We claim that $x^* x \in B$ implies
$\limi{n} x (a^* a)^{1/n} = x.$
To see this,
we note that $y \in B$ implies
$\limi{n} y (a^* a)^{1/n} = \limi{n} (a^* a)^{1/n} y = y.$
We now estimate:
\begin{align*}
\| x (a^* a)^{1/n} - x \|^2
& = \big\| [x (a^* a)^{1/n} - x]^* [x (a^* a)^{1/n} - x] \big\|
       \\
& = \big\| \big[ (a^* a)^{1/n} x^* x - x^* x \big] (a^* a)^{1/n}
          + \big[x^* x (a^* a)^{1/n} - x^* x \big] \big\|,
\end{align*}
which converges to zero as $n \to \I.$
This proves the claim.
Similarly $x^* x \in C$ implies
$\limi{n} x (a a^*)^{1/n} = x.$

We now prove~(\ref{L-Polar606_2Z10One}).
Suppose $x^* x \in B.$
Then
$x v^* = \limi{n} x (a^* a)^{1/n} v^*$
and $x (a^* a)^{1/n} v^* \in A$ for all $n \in \N.$
If $x x^* \in B,$
we get $x^* v^* \in A,$
so $v x \in A.$

We next consider~(\ref{L-Polar606_2Z10Three}).
By standard properties of the polar decomposition,
we have
\begin{equation}\label{Eq_2Z13_Star}
v^* a = (a^* a)^{1/2}.
\end{equation}
It follows that
\[
v^* v a^* a
 = v^* v (a^* a)^{1/2} (a^* a)^{1/2}
 = v^* a (a^* a)^{1/2}
 = a^* a.
\]
Taking adjoints, we get $a^* a v^* v = a^* a.$
For any \cfn{} $f \colon [0, \I) \to [0, \I)$
such that $f (0) = 0,$
one now gets
$f (a^* a) v^* v = f (a^* a)$
by polynomial approximation.
If now $x \in A$ satisfies $x^* x \in B,$
we use the claim in the second paragraph to get
\[
x v^* v
 = \limi{n} x (a^* a)^{1/n} v^* v
 = \limi{n} x (a^* a)^{1/n}
 = x.
\]
This is one half of~(\ref{L-Polar606_2Z10Three}).
The other half follows by applying this to $x^*$
and taking adjoints.


We next claim that
\begin{equation}\label{Eq_2Z13_CRel}
v (a^* a) = (a a^*) v.
\end{equation}
Taking adjoints in~(\ref{Eq_2Z13_Star}),
we get $a^* v = (a^* a)^{1/2}.$
So
\[
a a^* v
 = a (a^* a)^{1/2}
 = v (a^* a)^{1/2} (a^* a)^{1/2}
 = v a^* a,
\]
as desired.

We prove~(\ref{L-Polar606_2Z10Two}).
Elements of $B$ of the form
$x = (a^* a)^{3/2} y (a^* a)^{3/2},$ with $y \in A,$
are dense in~$B,$
so it suffices to consider such elements.
We have, using~(\ref{Eq_2Z13_CRel}) and its adjoint at the second step,
\begin{align*}
v x v^*
& = v (a^* a) (a^* a)^{1/2} y (a^* a)^{1/2} (a^* a) v^*
  \\
& = (a a^*) v (a^* a)^{1/2} y (a^* a)^{1/2} v^* (a a^*)
  = (a a^*) a y a^* (a a^*)
  \in C,
\end{align*}
as desired.

We now know that $\ph$ as in the statement is a map from
$B$ to~$C.$
That it is a \hm{} follows from
the relations $v^* v x = x v^* v = x$ for $x \in B,$
which are a consequence of~(\ref{L-Polar606_2Z10Three}).

Define $\ps \colon C \to A^{**}$
by $\ps (x) = v^* x v$ for $x \in C.$
The same relations imply that $(\ps \circ \ph) (x) = x$
for all $x \in B.$

We now claim that $\ps (C) \S B.$
By continuity,
it suffices to show that for $y \in A$
we have
$v^* (a a^*) y (a a^*) v \in B.$
Using~(\ref{Eq_2Z13_CRel}) and its adjoint at the first step,
we get
\begin{align*}
v^* (a a^*) y (a a^*) v
& = (a^* a)^{1/2} (a^* a)^{1/2} v^* y v (a^* a)^{1/2} (a^* a)^{1/2}
  \\
& = (a^* a)^{1/2} a^* y a (a^* a)^{1/2}
  \in B,
\end{align*}
as desired.
Thus, we can treat $\ps$ as a function from $C$ to~$B.$

We saw above that $\ps \circ \ph = \id_B.$
It follows from~(\ref{Eq_2Z13_Star})
that $v v^* a a^* = a a^*.$
Using the same reasoning
as in the proof of~(\ref{L-Polar606_2Z10Three}),
by taking adjoints we get $a a^* v v^* = a a^*,$
and by polynomial approximation we get
$(a a^*)^{1/n} v v^* = (a a^*)^{1/n}$ for all $n \in \N.$
For $x \in C$ we then use the last statement in the second paragraph
to get
$x v v^* = x$ for all $x \in C,$
and take adjoints to get $v v^* x = x$ for all $x \in C.$
These relations imply that $\ph \circ \ps = \id_C.$

It remains only to prove that $\ph (x) \sim x$ for $x \in B_{+}.$
Set $c = v x^{1/2}.$
We have $c \in A$ by~(\ref{L-Polar606_2Z10One}),
$c^* c = x$ by~(\ref{L-Polar606_2Z10Three}),
and $c c^* = \ph (x),$
so the result follows
from the discussion after Definition~2.3 of~\cite{KR}.
\end{proof}

We need only one of the implications of the following proposition,
but it seems informative to include the whole result.

\begin{prp}\label{P_3720_P3Her}
Let $A$ be a \ca{}
and let $D \S A$ be a subalgebra.
\Tfae:
\begin{enumerate}
\item\label{P_3720_P3Her_Cz}
For every $a \in A_{+} \SM \{ 0 \}$
there is $b \in D_{+} \SM \{ 0 \}$
such that $b \precsim a.$
\item\label{P_3720_P3Her_Cut}
For every $a \in A_{+} \SM \{ 0 \}$
there is $z \in A$ such that $z a z^*$ is a nonzero elenent of~$D.$
\item\label{P_3720_P3Her_HIso}
For every nonzero \hsa{} $B \S A$
there is a nonzero \hsa{} $E \S D$
and an injective homomorphism $\ph \colon E \to B$
such that for all $x \in E_{+}$ we have $\ph (x) \sim x.$
\end{enumerate}
\end{prp}

\begin{proof}
We show that (\ref{P_3720_P3Her_Cz}) implies~(\ref{P_3720_P3Her_Cut}).
Let $a \in A_{+} \SM \{ 0 \}.$
Choose $b \in D_{+} \SM \{ 0 \}$
such that $b \precsim a.$
\Wolog{} $\| b \| = 1.$
The definition of Cuntz subequivalence provides
$v \in A$ such that $\| v a v^* - b \| < \frac{1}{2}.$
Lemma 2.5(ii) of~\cite{KR} provides
$w \in A$ such that $\big( b - \frac{1}{2} \big)_{+} = w v a v^* w^*.$
The element $\big( b - \frac{1}{2} \big)_{+}$ is nonzero and in~$D.$
Take $z = w v.$

Now assume~(\ref{P_3720_P3Her_Cut});
we prove~(\ref{P_3720_P3Her_HIso}).
Choose any nonzero element
$a \in B_{+}.$
Choose $z \in A$ such that $z a z^*$ is a nonzero elenent of~$D.$
Set
\[
d = z a z^*,
\,\,\,\,\,\,
E = {\overline{d D d}},
\andeqn
b = a^{1/2} z^* z a^{1/2} \in B.
\]
The last part of Lemma~\ref{L-Polar606}
provides an isomorphism
$\ph \colon {\overline{d A d}} \to {\overline{b A b}} \subset B$
such that $x \sim \ph (x)$ for all
$x \in {\overline{d A d}}.$
The conclusion of the lemma follows by restricting
to ${\overline{d D d}}.$

Finally, we prove that (\ref{P_3720_P3Her_HIso})
implies~(\ref{P_3720_P3Her_Cz}).
Let $a \in A_{+} \SM \{ 0 \}.$
Set $B = {\overline{a A a}}.$
Let $E \S D$ and $\ph \colon E \to B$ be as in~(\ref{P_3720_P3Her_HIso}).
Choose any $b \in D_{+} \SM \{ 0 \}.$
Then,
using Proposition 2.7(i) of~\cite{KR} for the second step,
we have $b \sim \ph (b) \precsim a.$
\end{proof}

\begin{cor}\label{C-HerIso606}
Let $\af \colon G \to \Aut (A)$
be a pointwise spectrally nontrivial action
of a discrete group $G$ on a \ca~$A.$
Identify $A$ with a subalgebra of $C^*_{\mathrm{r}} (G, A, \af)$
in the usual way.
Let $B \subset C^*_{\mathrm{r}} (G, A, \af)$
be a nonzero hereditary subalgebra.
Then there exists a nonzero hereditary subalgebra $E \subset A$
and an injective homomorphism $\ph \colon E \to B$
such that, for every $x \in E_{+},$
we have
$x \sim \ph (x)$
in $C^*_{\mathrm{r}} (G, A, \af).$
\end{cor}

\begin{proof}
We apply Proposition~\ref{P_3720_P3Her}
with $C^*_{\mathrm{r}} (G, A, \af)$ in place of~$A$
and with $A$ in place of~$D.$
Lemma~\ref{L-L8_012Mod} states that (\ref{P_3720_P3Her_Cut}) holds,
and the conclusion is that (\ref{P_3720_P3Her_HIso}) holds.
\end{proof}

The fact that we get $x \sim \ph (x)$ is special to
the case in which we have some freeness condition.
In Proposition~\ref{P_2Z11_SubHer},
for an arbitrary action of~$\Z_2,$
we get all of the conclusion of Corollary~\ref{C-HerIso606}
except for this Cuntz equivalence.
For our applications,
it does not matter,
but it seems potentially useful.

\section{Actions of~$\Z_2$}\label{Sec:Start}

\indent
We conjecture the following analog of Corollary~\ref{C-HerIso606}
for finite groups, but with no condition on the action.
We omit the condition that for every $x \in D_{+},$
we have the Cuntz equivalence $x \sim \ph (x)$
in $C^* (G, A, \af).$
The example of $\Z_2$ acting trivially on~$\C$
shows that, in general,
this is not possible.

\begin{cnj}\label{C-609-FGpHIso}
Let $\af \colon G \to \Aut (A)$
be an action
of a finite group $G$ on a \ca~$A.$
Let $B \subset C^* (G, A, \af)$
be a nonzero hereditary subalgebra.
Then there exists a nonzero hereditary subalgebra $D \subset A$
and an injective homomorphism $\ph \colon D \to B.$
\end{cnj}

We prove this conjecture for $G = \Z_2.$
In fact, we prove a stronger result.
If $\af \colon \Z_2 \to \Aut (A)$
is any action on any \ca~$A,$
and if $D \S A$ is a nonzero hereditary subalgebra,
then there is a nonzero hereditary subalgebra $B$
of the fixed point algebra of $\af$
which is isomorphic to a (not necessarily hereditary)
subalgebra of~$D.$

There are two obstructions to generalization.
First,
we use equivariant semiprojectivity
of the cone over $C (\Z_2)$
with the translation action of $\Z_2$ on~$\Z_2.$
(See Lemma~\ref{L_2Z11_ConeEqPj}.)
The analogous statement for other groups is known
only for cyclic groups whose order is a power of~$2.$
(See Proposition~2.10 of~\cite{PST},
where equivariant projectivity is proved.)
Second, the combinatorics in proofs such as
that of Proposition~\ref{P_2Z11_SubHer}
would need to be much more complicated,
and may not give strong enough inequalities
to use as hypotheses in an analog of Lemma~\ref{L_2Z07_DistPos}.

\begin{ntn}\label{N_2Z07_Z2}
We systematically use the same letter for an
automorphism $\af$ of a \ca~$A$ such that $\af^2 = \id_A$
and the corresponding action $\af \colon \Z_2 \to \Aut (A).$
In particular,
when appropriate,
${\widehat{\af}} \in \Aut (C^* (\Z_2, A, \af))$
is the automorphism which generates the dual action.
We write the fixed point algebra as $A^{\af}.$
We also identify $A$ with its image in $C^* (\Z_2, A, \af)$
in the usual way.
\end{ntn}

\begin{lem}\label{C_2Z11_HerIso}
Let $A$ be a \ca, let $\ep > 0,$
and let $a, b \in A_{+}$ satisfy
$\| a - b \| < \ep.$
Then ${\overline{(b - \ep)_{+} A (b - \ep)_{+}}}$
is isomorphic to a hereditary subalgebra of ${\overline{a A a}}.$
\end{lem}

\begin{proof}
Lemma~2.5(ii) of~\cite{KR}
provides $v \in A$ such that $v a v^* = (b - \ep)_{+}.$
Set $c = a^{1/2} v^* v a^{1/2}.$
Then
${\overline{c A c}} \cong {\overline{(b - \ep)_{+} A (b - \ep)_{+}}}$
by the last part of Lemma~\ref{L-Polar606},
and ${\overline{c A c}}$
is a hereditary subalgebra of ${\overline{a A a}}.$
\end{proof}

\begin{lem}\label{L_2Z07LessThan1}
Let $A$ be a \ca,
let $\af \colon \Z_2 \to \Aut (A)$ be an action of~$\Z_2$ on~$A,$
and let $x \in A$ satisfy
\[
0 \leq x \leq 1,
\,\,\,\,\,\,
\| x \| = 1,
\andeqn
\| x - \af (x) \| < 1.
\]
Then there exists a $\af$-invariant hereditary subalgebra $E \S A$
which is isomorphic to a
nonzero hereditary subalgebra of ${\overline{x A x}}.$
\end{lem}

\begin{proof}
Set $a = \tfrac{1}{2} [x + \af (x)].$
Then
\[
\| a - x \| = \tfrac{1}{2} \| x - \af (x) \| < \tfrac{1}{2}.
\]
Therefore $\| a \| > \tfrac{1}{2}.$
Set $b = \big( a - \tfrac{1}{2} \big)_{+}$
and $E = {\overline{b A b}}.$
Then $E \neq 0,$
and is isomorphic to a hereditary subalgebra of ${\overline{x A x}}$
by Lemma~\ref{C_2Z11_HerIso}.
\end{proof}

\begin{dfn}\label{D_2Z07_Orth}
Let $A$ be a \ca{} and let $S, T \S A$ be selfadjoint subsets.
We say that $S$ and $T$ are
{\emph{orthogonal}}
if $a b = 0$ for all $a \in S$ and $b \in T.$
\end{dfn}

The definition is symmetric in $S$ and $T$
because of the requirement that $S$ and $T$ be selfadjoint.

\begin{lem}\label{L_2Z07_HerOrth}
Let $A$ be a \ca,
let $\af \colon \Z_2 \to \Aut (A)$ be an action of~$\Z_2$ on~$A,$
and let $x \in A$ be a nonzero element such that
$\af (x) = x^*.$
Let $D \S A$ be the hereditary subalgebra
generated by $x^* x.$
Suppose that $D$ and $\af (D)$ are orthogonal.
Let $F \S A$ be the hereditary subalgebra
generated by $D$ and $\af (D).$
Then $F$ is $\af$-invariant,
and there is an automorphism $\ps \in \Aut (D)$
such that $\ps^2 = \id_D$
and such that $D^{\ps}$ is isomorphic to a corner of $F \cap A^{\af}.$
\end{lem}

There is surely a generalization to actions
of arbitrary finite groups.
If $G$ is finite
and $\af \colon G \to \Aut (A)$ is an action,
then one should require that the hereditary subalgebras
$\af_g (D)$ be pairwise orthogonal
and that there exist an equivariant homomorphism as follows.
Define an action of $G$
on the cone $C_0 ( (0, 1]) \otimes C (G)$
by letting $G$ act on $G$ by translation
and letting $G$ act trivially on $C_0 ( (0, 1]).$
For $g \in G,$
let $p_g \in C (G)$ be the projection $p_g = \ch_{ \{ g \} }.$
Let $c \in C_0 ( (0, 1])$ be the function $c (t) = t$
for $t \in (0, 1].$
Then there should be an equivariant homomorphism
$\gm \colon C_0 ( (0, 1]) \otimes C (G) \to A$
such that $\af_g (D)$ is the hereditary subalgebra of~$A$
generated by $\gm (c \otimes p_g)$ for all $g \in G.$
(Actually, the last part follows if we merely assume
that $D$ is the hereditary subalgebra of~$A$
generated by $\gm (c \otimes p_1).$)

\begin{proof}[Proof of Lemma~\ref{L_2Z07_HerOrth}]
Let $\af^{**}$ be the automorphism of $A^{**}$
induced by $\af \in \Aut (A).$
Apply Lemma~\ref{L-Polar606} with $x$ in place of~$a,$
and let $v \in A^{**}$ be as there,
that is, the polar decomposition of $x$ is $x = v (x^* x)^{1/2},$
and $a \mapsto v a v^*$
is an isomorphism
from $D$ to ${\overline{(x x^*) A (x x^*)}} = \af (D).$
Then the polar decomposition of $x^*$ is $x^* = v^* (x x^*)^{1/2}.$
Therefore $\af^{**} (v) = v^*.$
Orthogonality of $x^* x$ and $x x^*$
implies that $v^2 = 0.$

Define $\mu \colon M_2 (D) \to A$ by
\[
\mu \left(  \left( \begin{matrix}
a_{1, 1} & a_{1, 2} \\
a_{2, 1} & a_{2, 2}
\end{matrix} \right)  \right)
= a_{1, 1} + a_{1, 2} v^* + v a_{2, 1} + v a_{2, 2} v^*
\]
for $a_{1, 1}, a_{1, 2}, a_{2, 1}, a_{2, 2} \in D.$
It follows from Lemma~\ref{L-Polar606}
that the range of~$\mu,$
which a priori is in~$A^{**},$
really is contained in~$A.$
A calculation using orthogonality of $D$ and $\af (D),$
using $a \in D$ implies $v a v^* \in \af (D)$
(Lemma~\ref{L-Polar606}(\ref{L-Polar606_2Z10Two})),
and using 
Lemma~\ref{L-Polar606}(\ref{L-Polar606_2Z10Three})
multiple times to both insert and remove factors of $v^* v,$
shows that $\mu$ is a \hm.

Identify $M_2 (D) = M_2 \otimes D,$
and let $(e_{j, k})_{j, k = 1, 2}$
be the standard system of matrix units for $M_2.$

The restriction of $\mu$ to $\C e_{1, 1} \otimes D$
is injective,
so $\mu$ is injective.

We claim that the range ${\operatorname{ran}} (\mu)$
of $\mu$ is equal to~$F.$
We first show that ${\operatorname{ran}} (\mu) \S F.$
Obviously $\mu ( \C e_{1, 1} \otimes D ) \S F$
and $\mu ( \C e_{2, 2} \otimes D ) = \af (D) \S F.$
For $a \in D,$
we have
\[
\mu (e_{1, 2} \otimes a)^* \mu (e_{1, 2} \otimes a)
 = v a^* a v^*
 \in \af (D)
 \S F
\]
and, using Lemma~\ref{L-Polar606}(\ref{L-Polar606_2Z10Three}),
\[
\mu (e_{1, 2} \otimes a) \mu (e_{1, 2} \otimes a)^*
 = a v^* v a^*
 = a a^*
 \in D
 \S F,
\]
so $\mu (e_{1, 2} \otimes a) \in F.$
Also,
$\mu (e_{2, 1} \otimes a) = \mu (e_{1, 2} \otimes a^*)^* \in F.$
Thus ${\operatorname{ran}} (\mu) \S F.$

We now show that $F \S {\operatorname{ran}} (\mu).$
Let $b \in F.$
Let $(e_{\ld})_{\ld \in \Ld}$
be an approximate identity for~$D.$
Using orthogonality of $D$ and $\af (D),$
we see that $D + \af (D)$ is a subalgebra of~$A,$
and that $\big( e_{\ld} + \af (e_{\ld}) \big)_{\ld \in \Ld}$
is an approximate identity for $D + \af (D)$
and therefore also for~$F.$
So it suffices to prove that
for $\ld \in \Ld$ we have
\[
e_{\ld} b e_{\ld}, \, e_{\ld} b \af (e_{\ld}), \,
\af (e_{\ld}) b e_{\ld}, \, \af (e_{\ld}) b \af (e_{\ld})
\in {\operatorname{ran}} (\mu).
\]
We have $e_{\ld} b e_{\ld} \in D \S {\operatorname{ran}} (\mu).$
Also
\[
\af (e_{\ld}) b \af (e_{\ld})
  \in \af (D)
  = v D v^*
  \S \mu (\C e_{2, 2} \otimes D)
  \S {\operatorname{ran}} (\mu).
\]
Next,
set $a = e_{\ld} b \af (e_{\ld}) v.$
Then $a \in A$ because $\af (e_{\ld}) v = \af (e_{\ld} v^*)$
and $e_{\ld} v^* \in A$
by Lemma~\ref{L-Polar606}(\ref{L-Polar606_2Z10One}).
Next,
\[
a a^*
 = e_{\ld} b \af (e_{\ld} v^* v e_{\ld}) b^* e_{\ld}
 \in D
\andeqn
a^* a
 = \af \big( v e_{\ld} \af( b^* e_{\ld}^2 b ) e_{\ld} v^* \big)
 \in \af (v D v^*)
 = D.
\]
Therefore $a \in D.$
So, using Lemma~\ref{L-Polar606}(\ref{L-Polar606_2Z10Three})
at the first step,
\[
e_{\ld} b \af (e_{\ld})
 = e_{\ld} b \af (e_{\ld} v^* v)
 = a v^*
 = \mu (e_{1, 2} \otimes a)
 \in {\operatorname{ran}} (\mu).
\]
Finally,
applying the case just done to~$b^*,$
we get
\[
\af (e_{\ld}) b e_{\ld}
 = \big( e_{\ld} b^* \af (e_{\ld}) \big)^*
 \in {\operatorname{ran}} (\mu).
\]
This completes the proof that ${\operatorname{ran}} (\mu) = F.$

Using Lemma~\ref{L-Polar606}(\ref{L-Polar606_2Z10Two}),
define $\ps \colon D \to D$
by $\ps (a) = \af (v a v^*) \in \af^2 (D) = D.$
We claim that $\ps^2 = \id_D.$
For $a \in D,$ we compute,
using Lemma~\ref{L-Polar606}(\ref{L-Polar606_2Z10Three})
at the last step,
\[
\ps^2 (a)
 = \af \big( v \af (v a v^*) v^* \big)
 = v^* \af^2 (v a v^*) v
 = v^* v a v^* v
 = a.
\]
This proves the claim.

By Takai duality (Theorem 7.9.3 of~\cite{Pd1}),
there is an isomorphism
\[
\rh \colon
 C^* \big( {\widehat{\Z_2}}, \, C^* (\Z_2, D, \ps), \, {\widehat{\ps}}
                              \big)
 \to M_2 (D)
\]
which is equivariant for the second dual action on the domain
and the $\sm$ action on $M_2 \otimes D$ given by conjugation by the
right regular representation on $M_2,$
tensored with~$\ps.$
The description of $\sm$ means that
\[
\sm \left( \left( \begin{matrix}
a_{1, 1} & a_{1, 2} \\
a_{2, 1} & a_{2, 2}
\end{matrix} \right) \right)
= \left( \begin{matrix}
\ps (a_{2, 2}) & \ps (a_{2, 1}) \\
\ps (a_{1, 2}) & \ps (a_{1, 1})
\end{matrix} \right)
\]
for $a_{1, 1}, a_{1, 2}, a_{2, 1}, a_{2, 2} \in D.$

We now claim that $\mu$ is equivariant for this action.
It is enough to check equivariance on a generating set for $M_2 (D).$
We choose
\[
\{ e_{1, 1} \otimes a, \, e_{1, 2} \otimes a \colon a \in D \}.
\]
We check, using the definition of $\ps$ at the second step in
both calculations,
and using Lemma~\ref{L-Polar606}(\ref{L-Polar606_2Z10Three}),
$\af^{**} (v) = v^*,$ and $\af^{**} (v^*) = v$
several times:
for $a \in D,$
\begin{align*}
(\mu \circ \sm) (e_{1, 1} \otimes a)
& = \mu (e_{2, 2} \otimes \ps (a))
 \\
& = v \af (v a v^*) v^*
  = \af ( (v^* v) a (v^* v) )
  = \af (a)
  = (\af \circ \mu) (e_{1, 1} \otimes a)
\end{align*}
and
\begin{align*}
(\mu \circ \sm) (e_{1, 2} \otimes a)
& = \mu (e_{2, 1} \otimes \ps (a))
 \\
& = v \af (v a v^*)
  = \af ( (v^* v) a v^*)
  = \af (a v^*)
  = (\af \circ \mu) (e_{2, 1} \otimes a).
\end{align*}
This proves the claim.

It follows that $F \cap A^{\af} = \mu ( M_2 (D)^{\sm}).$
(We have $M_2 (D)^{\sm} \cong C^* (\Z_2, D, \ps),$
but we won't use this fact.)
It is easy to see that
\[
M_2 (D)^{\sm} = \left\{ \left( \begin{matrix}
  a       &  b        \\
 \ps (b)  & \ps (a)
\end{matrix} \right)
    \colon a, b \in D \right\}
\]
Define a \pj{} $e \in M_2$ by
\[
e = \frac{1}{2} \left( \begin{matrix}
1 & 1 \\
1 & 1
\end{matrix} \right).
\]
Then $e$ is in the $\sm$-invariant part of the
multiplier algebra of $M_2 (D),$
so that $e M_2 (D)^{\sm} e$ is a corner
of $M_2 (D)^{\sm}.$
Computations show that there is an injective
\hm{} $\nu \colon D^{\ps} \to M_2 (D)^{\sm}$
given by
\[
\nu (c) = \frac{1}{2} \left( \begin{matrix}
c & c \\
c & c
\end{matrix} \right)
\]
for $c \in D^{\ps},$
that $e \nu (c) e = \nu (c)$
for $c \in D^{\ps},$
and that if $a, b \in D$ and we define
\[
c = \frac{1}{4} \big( a + b + \ps (a) + \ps (b) \big),
\]
then
\[
e \left( \begin{matrix}
  a       &  b        \\
 \ps (b)  & \ps (a)
\end{matrix} \right) e
= \nu (c).
\]
Therefore $\mu \circ \nu$ is an isomorphism from $D^{\ps}$
to the corner $\mu (e M_2 (D)^{\sm} e)$ of $F \cap A^{\af}.$
\end{proof}

\begin{lem}\label{L_2Z07_Avg}
Let $A$ be a \ca,
let $\af \colon \Z_2 \to \Aut (A)$ be an action of~$\Z_2$ on~$A,$
and let $D \S A$ be a hereditary subalgebra.
Suppose that there is a nonzero element $x \in A$
such that $x^* x \in D$ and $x x^* \in \af (D).$
Then there is a nonzero element $y \in A$
such that $y^* y \in D,$ $y y^* \in \af (D),$
and $\af (y) = y^*.$
\end{lem}

\begin{proof}
Suppose $\af (x) = - x^*.$
Set $y = i x.$ 
Then $y^* y = x^* x,$ $y y^* = x x^*,$
and $\af (y) = i (- x^*) = (i x)^* = y^*.$

So suppose $\af (x) \neq - x^*,$
and set $y = x + \af (x^*).$
Then $y \neq 0.$
We need the fact that if $B$ is a hereditary subalgebra of~$A,$
and $a, b \in A$ satisfy $a^* a, b^* b \in B,$
then $a^* b \in B.$
(To see this, use Proposition 1.4.5 of~\cite{Pd1}
to find $v, w \in A$
such that $a = v (a^* a)^{1/4}$ and $b = w (b^* b)^{1/4}.$
We have $(a^* a)^{1/4}, \, (b^* b)^{1/4} \in B,$
so $a^* b = (a^* a)^{1/4} v^* w (b^* b)^{1/4} \in B.$)
Now expand:
\[
y^* y
 = x^* x + x^* \af (x^*) + \af (x) x + \af (x) \af (x^*).
\]
We have $\af (x) \af (x^*) = \af (x x^*) \in \af^2 (D) = D.$
So all four terms have the form $a^* b$
with $a^* a, \, b^* b \in D.$
Therefore $y^* y \in D.$
Essentially the same argument shows that $y y^* \in \af (D).$
\end{proof}

\begin{lem}\label{L_2Z07_DistPos}
Let $A$ be a \ca,
and let $a, b \in A$ satisfy
\[
0 \leq a \leq 1,
\,\,\,\,\,\,
0 \leq b \leq 1,
\andeqn
\| a - b \| = 1.
\]
Then at least one of the following is true:
\begin{enumerate}
\item\label{L_2Z07_DistPos_bSuba}
For every $\ep > 0$ there is $x \in {\overline{a A a}}$
such that
\[
0 \leq x \leq 1,
\,\,\,\,\,\,
\| x \| = 1,
\,\,\,\,\,\,
\| x a - x \| < \ep,
\andeqn
\| x b \| < \ep.
\]
\item\label{L_2Z07_DistPos_aSubb}
For every $\ep > 0$ there is $x \in {\overline{b A b}}$
such that
\[
0 \leq x \leq 1,
\,\,\,\,\,\,
\| x \| = 1,
\,\,\,\,\,\,
\| x a \| < \ep,
\andeqn
\| x b - x \| < \ep.
\]
\setcounter{TmpEnumi}{\value{enumi}}
\end{enumerate}
\end{lem}

\begin{proof}
\Wolog{} $A$ is unital.
For any \ca~$A,$
let $S (A)$ denote the state space of~$A,$
and let $P (A)$ denote the pure state space of~$A.$
For any normal element $c \in A,$
by using the Hahn-Banach theorem
to extend from $C^* (1, c) \S A,$
we see that
$\| c \| = \sup_{\om \in S (A)} | \om (c) |.$
In particular,
$\| a - b \| = \sup_{\om \in S (A)} | \om (a) - \om (b) |.$
Since we always have $\om (a), \om (b) \in [0, 1],$
it follows that at least one of the following is true:
\begin{enumerate}
\setcounter{enumi}{\value{TmpEnumi}}
\item\label{L_2Z07_DPf_bSuba}
For all $\dt > 0$
there is $\om \in S (A)$ such that $\om (a) > 1 - \dt$
and $\om (b) < \dt.$
\item\label{L_2Z07_DPf_aSubb}
For all $\dt > 0$
there is $\om \in S (A)$ such that $\om (a) < \dt$
and $\om (b) > 1 - \dt.$
\end{enumerate}
We will prove that~(\ref{L_2Z07_DPf_bSuba})
implies conclusion~(\ref{L_2Z07_DistPos_bSuba})
in the statement of the lemma;
the same proof shows that~(\ref{L_2Z07_DPf_aSubb})
implies conclusion~(\ref{L_2Z07_DistPos_aSubb})
in the statement of the lemma.

First, (\ref{L_2Z07_DPf_bSuba}) clearly implies that $\| a \| = 1.$
Set
\[
\dt = \min \left( \frac{1}{32}, \, \frac{\ep^2}{242} \right).
\]
By linearity and the Krein-Milman Theorem,
(\ref{L_2Z07_DPf_bSuba})~implies that there is a pure state $\om$
on~$A$
such that
\[
\om (a) > 1 - \dt
\andeqn
\om (b) < \dt.
\]
Proposition~2.3 of~\cite{AAP}
implies that the state~$\om$
can be excised in the sense of Definition~2.1 of~\cite{AAP}.
In particular,
there is $y \in A$ such that
\[
0 \leq y \leq 1,
\,\,\,\,\,\,
\| y \| = 1,
\,\,\,\,\,\,
\| y a y - \om (a) y^2 \| < \dt,
\andeqn
\| y b y - \om (b) y^2 \| < \dt.
\]
It follows that
\[
\| y (1 - a) y \| = \| y a y - y^2 \| < 2 \dt
\andeqn
\| y b y \| < 2 \dt.
\]
Since $(1 - a)^2 \leq 1 - a,$
we get
\[
y (1 - a)^2 y \leq y (1 - a) y,
\]
whence $\| y (1 - a)^2 y \| < 2 \dt.$
Therefore
\begin{equation}\label{Eq_2Z07_Dst_y1a}
\| y (1 - a) \| = \| (1 - a) y \| < \sqrt{2 \dt}.
\end{equation}
Similarly $\| y b \| < \sqrt{2 \dt}.$
So
\begin{equation}\label{Eq_2Z07_Dst_yay}
\| a y a - y \|
 \leq \| a y - y \| \cdot \| a \| + \| y a - y \|
 < \sqrt{2 \dt} + \sqrt{2 \dt}
 = 2 \sqrt{2 \dt}.
\end{equation}
Therefore
\begin{equation}\label{Eq_2Z07_Dst_yAndb}
\| a y a b \|
 \leq \| a y a - y \| \cdot \| b \| + \| y b \|
 < 2 \sqrt{2 \dt} + \sqrt{2 \dt}
 = 3 \sqrt{2 \dt}.
\end{equation}

Using~(\ref{Eq_2Z07_Dst_yay}) at the second step and
$\dt \leq \frac{1}{32}$ at the fourth step,
we get
\[
1 \geq \| a y a \|
  > \| y \| - 2 \sqrt{2 \dt}
  = 1 - 2 \sqrt{2 \dt}
  \geq \tfrac{1}{2}.
\]
Therefore
\[
1 \leq  \frac{1}{\| a y a \|}
  \leq 1 + 2 (1 - \| a y a \|)
  < 1 + 4 \sqrt{2 \dt}.
\]
Setting $x = \| a y a \|^{-1} a y a,$
we then get
\begin{equation}\label{Eq_3113_Dst_xaya}
\| x - a y a \| < 4 \sqrt{2 \dt},
\end{equation}
so
%
\[
\| x - y \|
  \leq  \| x - a y a \| + \| a y a - y \|
  < 4 \sqrt{2 \dt} + 2 \sqrt{2 \dt}
  = 6 \sqrt{2 \dt}
\]
%
and (using~(\ref{Eq_2Z07_Dst_y1a}))
%
\[
\| y a - x \|
  \leq \| y - a y \| \cdot \| a \| + \| a y a - x \|
  < \sqrt{2 \delta} + 4 \sqrt{2 \delta}
  = 5 \sqrt{2 \delta}.
\]
%
Combining the last two inequalities
and using $\| a \| \leq 1,$
we get
\[
\| x a - x \|
  \leq \| x - y \| + \| y a - x \|
  < 6 \sqrt{2 \dt} + 5 \sqrt{2 \dt}
  = 11 \sqrt{2 \dt}
  \leq \ep.
\]
Combining (\ref{Eq_3113_Dst_xaya}) with~(\ref{Eq_2Z07_Dst_yAndb})
and using $\| b \| \leq 1,$
we get
\[
\| x b \|
  \leq \| x - a y a \| + \| a y a b \|
  < 4 \sqrt{2 \dt} + 3 \sqrt{2 \dt}
  = 7 \sqrt{2 \dt}
  \leq \ep.
\]
This completes the proof.
\end{proof}

\begin{lem}\label{L_2Z11_ConeEqPj}
Let $m \in \N$ and let $n = 2^m.$
For every $\ep > 0$ there is $\dt > 0$ such that the following holds.
Let $A$ be a \ca,
let $\af \in \Aut (A)$ satisfy $\af^n = \id_A,$
and let $a_1, a_2, \ldots, a_{n} \in A$
safisfy
$0 \leq a_k \leq 1$ for $k = 1, 2, \ldots, n,$
$\| \af (a_k) - a_{k + 1} \| < \dt$
for $k = 1, 2, \ldots, n$ (with $a_{n + 1} = a_1$),
and $\| a_j a_k \| < \dt$ for $j, k \in \{ 1, 2, \ldots, n \}$
with $j \neq k.$
Then there exist $b_1, b_2, \ldots, b_{n} \in A$
such that for $k = 1, 2, \ldots, n$ we have
(taking $b_{n + 1} = b_1$)
\[
0 \leq b_k \leq 1,
\,\,\,\,\,\,
\af (b_k) = b_{k + 1},
\andeqn
\| b_k - a_k \| < \ep,
\]
and such that $b_j b_k = 0$ for $j, k \in \{ 1, 2, \ldots, n \}$
with $j \neq k.$
\end{lem}

The case $n = 2$
(which is what we need) has a fairly easy direct proof.
Set $x = a_1 - a_2.$
Then $\| x + \af (x) \| < 2 \dt.$
The element $y = \frac{1}{2} (x - \af (x))$
satisfies $\af (y) = - y$
and $\| y - x \| < \dt.$
Take $b_1$ to be the positive part of~$y$
and take $b_2$ to be the negative part of~$y.$

\begin{proof}[Proof of Lemma~\ref{L_2Z11_ConeEqPj}]
We write $\af \colon \Z_n \to \Aut (A)$
for the action of $\Z_n$ generated by~$\af.$
Replacing $A$ by the smallest $\af$-invariant subalgebra
containing $a_1, a_2, \ldots, a_{n},$
we may assume that $A$ is separable.

Let $C = C_0 ( (0, 1]) \otimes C (\Z_n)$
be the cone over $C (\Z_n),$
with the action of $\Z_n$
obtained by letting $\Z_n$ act on $\Z_n$ by translation
and letting $\Z_n$ act trivially on $C_0 ( (0, 1]).$
Then $C$ is equivariantly projective by Proposition~2.10
of~\cite{PST}.
Therefore $C$ is equivariantly semiprojective.
It follows from Lemma~1.5 of~\cite{PST}
that the unitization $C^{+}$ of~$C$
is equivariantly semiprojective in the unital category.

The algebra~$C^{+},$ with the specified action of~$\Z_n,$
is given by the following
$\Z_n$-equivariant set $(S, \sm, R)$ of generators and relations
in the sense of Definition~5.8 of~\cite{Ph-EqSj}.
We take $S = \{ x_1, x_2, \ldots, x_{n} \},$
we take $\sm$ to be the action of $\Z_n$ on $S$ generated
by the cyclic permutation $x_k \mapsto x_{k + 1}$
for $k = 1, 2, \ldots, n$ (with $x_{n + 1} = x_1$),
and we take $R$ to consist of the relations
$0 \leq x_k \leq 1$ for $k = 1, 2, \ldots, n$
and $x_j x_k = 0$ for $j, k \in \{ 1, 2, \ldots, n \}$
with $j \neq k.$
(See Remark~5.6 of~\cite{Ph-EqSj}
for the justification of these as relations.)
The pair $(S, R)$ is finite, admissible, and bounded
(Definition~5.2 of~\cite{Ph-EqSj}).
Apply Theorem~5.22 of~\cite{Ph-EqSj}
to conclude that $(S, \sm, R)$ is stable in the sense
of Definition~5.20 of~\cite{Ph-EqSj}.
It is immediate from this definition that for every
$\ep > 0$ there is $\dt > 0$
such that for any \ca~$A$
the conclusion of the lemma holds in the unitization~$A^{+}$
in place of~$A.$

So assume that $a_1, a_2, \ldots, a_{n} \in A,$
and let $b_1, b_2, \ldots, b_{n} \in A^{+}$
be the resulting elements.
Let $\kp \colon A^{+} \to \C$ be the map coming from the unitization.
Since $b_j b_k = 0$ for $j \neq k,$
we have $\kp (b_j) \kp (b_k) = 0$ for $j \neq k.$
Since the induced action of $\Z_n$ fixes $1 \in A^{+},$
we have $\kp (b_{k + 1}) = \kp (\af (b_k)) = \kp (b_k)$
for $k = 1, 2, \ldots, n - 1.$
These two collections of relations are compatible
only if $\kp (b_k) = 0$ for $k = 1, 2, \ldots, n.$
Therefore $b_1, b_2, \ldots, b_{n}$ are in fact in~$A.$
\end{proof}

\begin{prp}\label{P_2Z11_SubHer}
Let $A$ be a \ca,
let $\af \colon \Z_2 \to \Aut (A)$ be an action of~$\Z_2$ on~$A,$
and let $D \S A$ be a nonzero hereditary subalgebra.
Then there exists a nonzero hereditary subalgebra $B \S A^{\af}$
which is isomorphic to a (not necessarily hereditary)
subalgebra of~$D.$
\end{prp}

\begin{proof}
We divide the proof into several cases.
At any stage,
we may replace $D$ by any nonzero hereditary subalgebra of~$D,$
and then perhaps appeal to a case already done.



{\textbf{Case~1:}} $D$ and $\af (D)$ are orthogonal,
and there is no nonzero $x \in A$ such that
$x^* x \in D$ and $x x^* \in \af (D).$

Orthogonality implies that $D + \af (D)$ is a subalgebra of~$A$
which is equivariantly isomorphic to $D \oplus D$
with the action generated by $\bt (a, b) = (b, a)$
for $a, b \in D.$
We claim that $D + \af (D)$ is a hereditary subalgebra of~$A.$
So let $y \in A$ satisfy $y^* y, \, y y^* \in D + \af (D).$
Let $(e_{\ld})_{\ld \in \Ld}$
be an approximate identity for~$D.$
Then $\big( e_{\ld} + \af (e_{\ld}) \big)_{\ld \in \Ld}$
is an approximate identity for $D + \af (D)$
and therefore also for the hereditary subalgebra
generated by $D + \af (D).$
For $\ld \in \Ld,$
the element $x = \af (e_{\ld}) y e_{\ld}$
satisfies $x^* x \in D$ and $x x^* \in \af (D).$
Therefore $\af (e_{\ld}) y e_{\ld} = 0.$
Similarly $\af (e_{\ld}) y^* e_{\ld} = 0,$
so $e_{\ld} y \af (e_{\ld}) = 0.$
Therefore
\[
y = \lim_{\ld \in \Ld}
      [ e_{\ld} + \af (e_{\ld}) ] y [ e_{\ld} + \af (e_{\ld}) ]
  = \lim_{\ld \in \Ld}
      \big( e_{\ld} y e_{\ld} + \af (e_{\ld}) y \af (e_{\ld}) \big)
  \in D + \af (D).
\]
The claim is proved.

Now $[D + \af (D)]^{\af}$ is a hereditary subalgebra of~$A^{\af},$
and $[D + \af (D)]^{\af} \cong D.$
This completes the proof of Case~1.

{\textbf{Case~2:}} $D$ and $\af (D)$ are orthogonal,
and there is a nonzero element $x \in A$ such that
$x^* x \in D$ and $x x^* \in \af (D).$

By Lemma~\ref{L_2Z07_Avg},
we may assume that $\af (x) = x^*.$
Therefore $\af (x^* x) = x x^*.$
Since we can replace $D$ by a nonzero hereditary subalgebra of~$D,$
we may assume that $D = {\overline{(x^* x) A (x^* x)}}.$
So $\af (D) = {\overline{(x x^*) A (x x^*)}}.$
Let $F \S A$ be the hereditary subalgebra
generated by $D$ and $\af (D).$
We apply Lemma~\ref{L_2Z07_HerOrth}
to get an automorphism $\ps \in \Aut (D)$
such that $\ps^2 = \id_D$
and such that $D^{\ps}$ is isomorphic to a corner of $F \cap A^{\af}.$
Then $D^{\ps}$ is a nonzero subalgebra of~$D$
which is isomorphic to a hereditary subalgebra of $A^{\af}.$

{\textbf{Case~3:}} $D$ and $\af (D)$ are not orthogonal.
Then there exist $a \in D$ and $b \in \af (D)$
such that $b a \neq 0.$
Set $c = b a.$
Then $c^* c \in D$ and $c c^* \in \af (D).$
By Lemma~\ref{L_2Z07_Avg},
there is $x \in A$
such that
\[
x \neq 0,
\,\,\,\,\,\,
x^* x \in D,
\,\,\,\,\,\,
x x^* \in \af (D),
\andeqn
\af (x) = x^*.
\]
Then $\af (x^* x) = x x^*.$
We may clearly assume that $\| x \| = 1.$
Since we can replace $D$ by a nonzero hereditary subalgebra of~$D,$
we may assume that $D = {\overline{(x^* x) A (x^* x)}}.$
So $\af (D) = {\overline{(x x^*) A (x x^*)}}.$

Suppose now that $\| x^* x - x x^* \| < 1.$
Apply Lemma~\ref{L_2Z07LessThan1}
with $x^* x$ in place of~$x.$
We obtain a nonzero $\af$-invariant hereditary subalgebra
$E \S A$ which is isomorphic to a
nonzero hereditary subalgebra of~$D.$
Then $E^{\af} = E \cap A^{\af}$
is a nonzero hereditary subalgebra of $A^{\af}$
which is isomorphic to a subalgebra of~$D.$

Otherwise,
we have $\| x^* x - x x^* \| \geq 1.$
In fact, we must have $\| x^* x - x x^* \| = 1.$

Use Lemma~\ref{L_2Z11_ConeEqPj}
to choose $\dt > 0$
such that whenever $y \in A$ satisfies $0 \leq y \leq 1$
and $\| y \af (y) \| < \dt,$
then there exists $z \in A$ such that
\begin{equation}\label{Eq_2Z11_Forz}
0 \leq z \leq 1,
\,\,\,\,\,\,
z \af (z) = 0,
\andeqn
\| z - y \| < \tfrac{1}{2}.
\end{equation}
Apply Lemma~\ref{L_2Z07_DistPos}
with $a = x^* x$ and $b = x x^*.$
We may exchange $x$ and $x^*$ if desired.
(This exchanges $D$ and $\af (D).$
But finding a nonzero hereditary subalgebra of $A^{\af}$
isomorphic to a subalgebra of $D$
is equivalent to finding a nonzero hereditary subalgebra of $A^{\af}$
isomorphic to a subalgebra of $\af (D).$)
We may therefore assume that from Lemma~\ref{L_2Z07_DistPos}
we obtain
$y \in D$
such that
\[
y \in D,
\,\,\,\,\,\,
0 \leq y \leq 1,
\,\,\,\,\,\,
\| y \| = 1,
\,\,\,\,\,\,
\| y (x^* x) - y \| < \frac{\dt}{2},
\andeqn
\| y (x x^*) \| < \frac{\dt}{2}.
\]
Then also $\| (x^* x) y - y \| < \frac{\dt}{2}.$

Now
\begin{align*}
\| y \af (y) \|
& \leq \| y \| \cdot \| \af (y) - (x x^*) \af (y) \|
     + \| y (x x^*) \| \cdot \| \af (y) \|
     \\
& = \| \af (y - (x^* x) y) \| + \| y (x x^*) \|
  < \frac{\dt}{2} + \frac{\dt}{2}
  = \dt.
\end{align*}
By the choice of~$\dt,$
there exists $z \in A$ satisfying~(\ref{Eq_2Z11_Forz}).
Since $\| y \| = 1,$
we have $\| z \| > \tfrac{1}{2},$
and therefore $c = \big( z - \tfrac{1}{2} \big)_{+}$
is nonzero.
Clearly $c \af (c) = 0.$
Let $E = {\overline{c A c}}.$
Then $E$ and $\af (E)$ are orthogonal hereditary subalgebras of~$A.$
By Case~1 or Case~2, as appropriate,
there exists a nonzero hereditary subalgebra $B \S A^{\af}$
which is isomorphic to a subalgebra of~$E.$
Lemma~\ref{C_2Z11_HerIso} implies that~$E,$
and hence~$B,$
is isomorphic to a subalgebra
of ${\overline{y A y}} \S D.$
\end{proof}

The following lemma is surely well known.

\begin{lem}\label{L_3112_ImHered}
Let $A$ and $B$ be \ca{s}
and let $\pi \colon A \to B$ be a surjective \hm.
Let $D \S A$ be a \hsa.
Then $\pi (D)$ is a \hsa{} of~$B.$
\end{lem}

\begin{proof}
Let $x \in \pi (D)$ and $y \in B$ satisfy $0 \leq y \leq x.$
Then $\limi{n} x^{1/n} y x^{1/n} = y.$
Choose $a \in D_{+}$ and $b \in A_{+}$
such that $\pi (a) = x$ and $\pi (b) = y.$
Then the elements $a^{1/n} b a^{1/n}$ are in $D$
and $\limi{n} \pi \big( a^{1/n} b a^{1/n} \big) = y.$
Therefore $y \in {\overline{\pi (D)}} = \pi (D).$
\end{proof}

\begin{lem}\label{L_3112_FixQ}
Let $A$ be a \ca{} and let $\af$ be an action of $\Z_2$ on~$A.$
Let $I \S A$ be an $\af$-invariant ideal.
Then $I^{\af}$ is an ideal in $A^{\af},$
and $A^{\af} / I^{\af}$ is isomorphic to the fixed point algebra
of the induced action of $\Z_2$ on~$A / I.$
\end{lem}

\begin{proof}
This is a special case of Lemma~1.6 of~\cite{Ph-EqSj}.
\end{proof}

\begin{prp}\label{P_3112_SubHer}
Let $A$ be a \ca,
let $\af$ be an action of~$\Z_2$ on~$A,$
let $I \S A$ be an ideal (not necessarily $\af$-invariant),
and let $D \S A / I$ be a nonzero hereditary subalgebra.
Then there exists an ideal $J \S A^{\af}$
and a nonzero hereditary subalgebra $B \S A^{\af} / J$
which is isomorphic to a (not necessarily hereditary)
subalgebra of~$D.$
\end{prp}

\begin{proof}
We divide the proof into several cases,
some of which will be done by reduction to previous cases,
possibly for a different choice of $A$ and~$I.$
We let $\pi \colon A \to A / I$ be the quotient map,
and we set $E = \pi^{-1} (D),$
which is a \hsa{} of~$A$ such that $I \S E.$

{\textbf{Case~1:}}
The ideal $I$ is $\af$-invariant.

There is an induced action
${\overline{\af}}$ of $\Z_2$ on $A / I.$
By Proposition~\ref{P_2Z11_SubHer},
there is a subalgebra $B \S D$
which is isomorphic to a nonzero \hsa{}
of $(A / I)^{\overline{\af}}.$
Set $J = I^{\af}.$
Then $(A / I)^{\overline{\af}} \cong A^{\af} / J$
by Lemma~\ref{L_3112_FixQ}.
This proves Case~1.

{\textbf{Case~2:}}
There is a \ca~$C$ such that $A = C \oplus C$
with $\af (x, y) = (y, x)$ for $x, y \in C,$
and $I = \{ 0 \} \oplus C.$
We claim that in this case, we can take $J = \{ 0 \}.$

We have $A / I \cong C,$
and there is an isomorphism $\ph \colon C \to A^{\af}$
given by $\ph (x) = (x, x)$ for all $x \in C.$
So $D$ is isomorphic to a \hsa{} of~$C,$
and thus $D$ is isomorphic to a \hsa{} of~$A^{\af}.$
This proves Case~2 with $J = \{ 0 \}.$

{\textbf{Case~3:}}
$I \cap \af (I) = \{ 0 \}$
and $I$ is a proper subset of $E \cap (I + \af (I)).$

Set $F = E \cap (I + \af (I)).$
Set $C = (I + \af (I)) / I,$
let $\kp \colon I + \af (I) \to C$
be the quotient map,
and define $\ph \colon I + \af (I) \to C \oplus C$
by $\ph (a) = (\kp (a), \, (\kp \circ \af) (a) )$
for $a \in I + \af (I).$
Then $\ph$ is bijective since $I \cap \af (I) = \{ 0 \}.$
Let $\gm \colon C \oplus C \to C \oplus C$
be the automorphism given by $\gm (x, y) = (y, x)$ for $x, y \in C.$
Using the action of $\Z_2$ that $\gm$ generates,
$\ph$ becomes equivariant.

We have $\ph (I) = \{ 0 \} \oplus C$
and $\ph (F) = \kp (F) \oplus C.$
The hypotheses of this case imply that $\kp (F) \neq \{ 0 \}.$
Lemma~\ref{L_3112_ImHered}
implies that $\kp (F)$ is a hereditary subalgebra of $C \S A / I,$
and clearly $\kp (F) \S D.$
Case~2 implies that $\kp (F)$ contains a subalgebra
isomorphic to a nonzero hereditary subalgebra of $I + \af (I),$
and hence isomorphic to a nonzero hereditary subalgebra of~$A.$
This proves Case~3.

{\textbf{Case~4:}}
$I \cap \af (I) = \{ 0 \}$
and $I$ is not a proper subset of $E \cap (I + \af (I)).$

Since $I \S E \cap (I + \af (I)),$
the hypotheses of this case imply that $E \cap (I + \af (I)) = I.$
Let $\ps \colon A / I \to A / (I + \af (I))$ be the quotient map.
Then $\ps |_D$ is injective.
Therefore $D$ is isomorphic to a nonzero hereditary subalgebra
of $A / (I + \af (I)).$
Since $I + \af (I)$ is $\af$-invariant,
the result in Case~4 follows from Case~1.

{\textbf{Case~5:}}
None of the previous cases applies.
Thus $I$ is not $\af$-invariant and $I \cap \af (I) \neq \{ 0 \}.$
Let
\[
\et \colon A \to A / (I \cap \af (I))
\andeqn
\mu \colon A / (I \cap \af (I)) \to A / I
\]
be the quotient maps.
Set $L = \et (I).$
Let ${\overline{\af}}$ be the induced action
of $\Z_2$ on $A / (I \cap \af (I)).$
Then $L \cap {\overline{\af}} (L) = \{ 0 \}$
and $\mu$ induces an isomorphism
$[A / (I \cap \af (I))] / L \to A / I.$
Therefore Case~3 or Case~4 applies with
$A / (I \cap \af (I))$ in place of~$A$
and $L$ in place of~$I.$
Thus there is an ideal $J_0 \S [A / (I \cap \af (I))]^{\overline{\af}}$
and a nonzero hereditary subalgebra
$B \S [A / (I \cap \af (I))]^{\overline{\af}} / J_0$
which is isomorphic to a (not necessarily hereditary)
subalgebra of~$D.$
We have
\[
[A / (I \cap \af (I))]^{\overline{\af}}
  \cong A^{\af} / (I \cap \af (I))^{\af}
\]
by Lemma~\ref{L_3112_FixQ}.
Let $J$ be the inverse image of $J_0$ in
$A^{\af}.$
Then $A^{\af} / J \cong [A / (I \cap \af (I))]^{\overline{\af}} / J_0,$
so we can identify $B$ with a \hsa{} of $A^{\af} / J.$
This completes the proof of Case~5.
\end{proof}

\section{Permanence for properties defined in terms of hereditary
  subalgebras}\label{Sec_HSAPerm}

\indent
To avoid repetition,
we present an abstract theory
which gives permanence results
for the properties we consider.
Most of
Sections \ref{Sec:PermHInf}, \ref{Sec:PermQSP}, and~\ref{Sec:WIP}
consists of applications of this theory.

\begin{dfn}\label{D_3719_UpClass}
Let $\CC$ be a class of \ca{s}.
We say that $\CC$ is
{\emph{upwards directed}} if
whenever $A$ is a \ca{} which contains a subalgebra isomorphic to
an algebra in~$\CC,$
then $A \in \CC.$
\end{dfn}

\begin{dfn}\label{D_3719_ResHerC}
Let $\CC$ be an upwards directed class of \ca{s},
and let $A$ be a \ca.
\begin{enumerate}
\item\label{D_3719_ResHerC_Her}
We say that $A$ is
{\emph{hereditarily in~$\CC$}}
if every nonzero \hsa{} of~$A$ is in~$\CC.$
\item\label{D_3719_ResHerC_Res}
We say that $A$ is
{\emph{residually hereditarily in~$\CC$}}
if $A / I$ is hereditarily in~$\CC$ for every ideal $I \S A$
with $I \neq A.$
\end{enumerate}
\end{dfn}

To show the usefulness of this concept now,
we point out that a \ca{} hereditarily contains a nonzero \pj{}
\ifo{} it has Property~(SP),
and that a \ca{}
residually hereditarily contains an infinite \pj{}
\ifo{} it is purely infinite and has the ideal property
(see~Proposition~\ref{P-609IdealPI} below).
In this section,
we give permanence results for the classes of
algebras which are (residually) hereditarily in such a class~$\CC.$

\begin{thm}\label{T_3719_CCSpFree}
Let $\CC$ be an upwards directed class of \ca{s}.
Let $\af \colon G \to \Aut (A)$
be a pointwise spectrally nontrivial action
of a discrete group $G$ on a \ca~$A.$
\begin{enumerate}
\item\label{T_3719_CCSpFree_Ntr}
If $\af$ is pointwise spectrally nontrivial
and $A$ is hereditarily in~$\CC,$
then $C^*_{\mathrm{r}} (G, A, \af)$ is hereditarily in~$\CC.$
\item\label{T_3719_CCSpFree_SF}
If $\af$ is exact and spectrally free,
and $A$ is residually hereditarily in~$\CC,$
then $C^*_{\mathrm{r}} (G, A, \af)$ is residually hereditarily in~$\CC.$
\end{enumerate}
\end{thm}

\begin{proof}
Part~(\ref{T_3719_CCSpFree_Ntr})
is immediate from Corollary~\ref{C-HerIso606}.

We prove~(\ref{T_3719_CCSpFree_SF}).
Let $J$ be an ideal in $C^*_{\mathrm{r}} (G, A, \af).$
By Proposition~\ref{P-SNAQ-Sep608},
there is an $\alpha$-invariant ideal $I$ in~$A$
such that $J = C^*_{\mathrm{r}} (G, I, \af).$
Exactness of the action
implies that
$C^*_{\mathrm{r}} (G, A, \af) / J \cong C^*_{\mathrm{r}} (G, A/I, \af).$
By hypothesis,
$A / I$ is hereditarily in~$\CC$
and the action of $G$ on $A / I$
is spectrally nontrivial,
so $C^*_{\mathrm{r}} (G, A, \af) / J$ is hereditarily in~$\CC$
by part~(\ref{T_3719_CCSpFree_Ntr}).
\end{proof}

\begin{cor}\label{C_3719_CCZRP}
Let $\CC$ be an upwards directed class of \ca{s}.
Let $A$ be a unital \ca,
and let $\af \in \Aut (A)$ have the Rokhlin property.
\begin{enumerate}
\item\label{C_3719_CCZRP_Her}
If $A$ is hereditarily in~$\CC,$
then $C^* (\Z, A, \alpha)$ is hereditarily in~$\CC.$
\item\label{C_3719_CCZRP_Res}
If $A$ is residually hereditarily in~$\CC,$
then $C^* (\Z, A, \alpha)$ is residually hereditarily in~$\CC.$
\end{enumerate}
\end{cor}

\begin{proof}
In view of Proposition~\ref{P-611-RPInpSFr},
this follows from Theorem~\ref{T_3719_CCSpFree}.
\end{proof}

\begin{thm}\label{T_3719_CCZ2}
Let $\CC$ be an upwards directed class of \ca{s}.
Let $\af \colon \Z_2 \to \Aut (A)$
be an arbitrary action of $\Z_2$ on a \ca~$A.$
\begin{enumerate}
\item\label{T_3719_CCZ2_Her}
If $A^{\af}$ is hereditarily in~$\CC,$
then $A$ is hereditarily in~$\CC.$
\item\label{T_3719_CCZ2_Res}
If $A^{\af}$ is residually hereditarily in~$\CC,$
then $A$ is residually hereditarily in~$\CC.$
\end{enumerate}
\end{thm}

\begin{proof}
Part~(\ref{T_3719_CCZ2_Her})
is immediate from Proposition~\ref{P_2Z11_SubHer}.
Part~(\ref{T_3719_CCZ2_Res})
is immediate from Proposition~\ref{P_3112_SubHer}.
\end{proof}

\begin{cor}\label{C_3719_CCZ2CP}
Let $\CC$ be an upwards directed class of \ca{s}.
Let $\af \colon \Z_2 \to \Aut (A)$
be an arbitrary action of $\Z_2$ on a \ca~$A.$
\begin{enumerate}
\item\label{C_3719_CCZ2CP_Her}
If $A$ is hereditarily in~$\CC,$
then $C^* (\Z_2, A, \alpha)$ is hereditarily in~$\CC.$
\item\label{C_3719_CCZ2CP_Res}
If $A$ is residually hereditarily in~$\CC,$
then $C^* (\Z_2, A, \alpha)$ is residually hereditarily in~$\CC.$
\end{enumerate}
\end{cor}

\begin{proof}
Apply Theorem~\ref{T_3719_CCZ2}
with $C^* (\Z_2, A, \alpha)$ in place of $A$
and the dual action ${\widehat{\af}}$ in place of~$\af.$
\end{proof}

The same theory gives some permanence results
not involving crossed products.

\begin{prp}\label{L-2Of3_CC-428}
Let $\CC$ be an upwards directed class of \ca{s}.
Let $A$ be a \ca,
and let $J \subset A$ be an ideal.
Then $A$ is residually hereditarily in~$\CC$
\ifo\  $J$ and $A / J$ are both residually hereditarily in~$\CC.$
\end{prp}

\begin{proof}
It is obvious that if $A$ is residually hereditarily in~$\CC$
then $J$ and $A / J$ are both residually hereditarily in~$\CC.$
So assume that $J$ and $A / J$
are residually hereditarily in~$\CC,$
let $I \subset A$ be an ideal,
and let $B \subset A / I$ be a nonzero hereditary subalgebra.
Set $C = B \cap [ (I + J) / I],$
which is a hereditary subalgebra of $A / I.$

Suppose first that $C \neq \{ 0 \}.$
Then $C$ is a nonzero hereditary subalgebra
of $(I + J) / I \cong J / (J \cap I).$
Since $J$ is residually hereditarily in~$\CC,$
we get $C \in \CC,$
whence $B \in \CC$ because $\CC$ is upwards directed.

Now suppose that $C = \{ 0 \}.$
Let $\pi \colon A / I \to A / (I + J)$
be the quotient map.
Then $\pi |_B$ is injective.
Therefore $B$ is isomorphic to a nonzero hereditary subalgebra
of $[ A / J] / [ (I + J) / J ].$
Since $A / J$ is residually hereditarily in~$\CC,$
we get $B \in \CC.$
\end{proof}

\begin{prp}\label{P-609_CCQHILim}
Let $\CC$ be an upwards directed class of \ca{s}.
Let $(A_{\ld})_{\ld \in \Ld}$ be a direct system of \ca{s}
with maps $\ph_{\mu, \ld} \colon A_{\ld} \to A_{\mu}$
for $\ld, \mu \in \Ld$ satisfying $\ld \leq \mu,$
with direct limit $A = \dirlim_{\ld} A_{\ld},$
and
with canonical maps $\ph_{\ld} \colon A_{\ld} \to A$ for $\ld \in \Ld.$
\begin{enumerate}
\item\label{P-609_CCQHILim-HI}
Suppose that $A_{\ld}$ is hereditarily in~$\CC$
for all $\ld \in \Ld$
and that $\ph_{\mu, \ld}$ is injective
for all $\ld, \mu \in \Ld$ satisfying $\ld \leq \mu.$
Then $A$ is hereditarily in~$\CC.$
\item\label{P-609_CCQHILim-Q}
Suppose that $A_{\ld}$ is residually hereditarily in~$\CC$
for all $\ld \in \Ld.$
Then $A$ is residually hereditarily in~$\CC.$
\end{enumerate}
\end{prp}

\begin{proof}
We prove~(\ref{P-609_CCQHILim-HI}).
We may assume that the algebras $A_{\ld}$ are all subalgebras of~$A,$
with $A_{\ld} \subset A_{\mu}$
for $\ld, \mu \in \Ld$ satisfying $\ld \leq \mu,$
and that $A = {\overline{\bigcup_{\ld \in \Ld} A_{\ld} }}.$

Let $B \subset A$ be a nonzero hereditary subalgebra.
Choose $z \in B_{+}$ such that $\| z \| = 1.$
Choose $\ld \in \Ld$ and $y \in (A_{\ld})_{+}$ such that
$\| z - y \| < \tfrac{1}{3}.$
Set $x = \big( y - \tfrac{1}{3} \big)_{+} \in (A_{\ld})_{+}.$
Then $x \neq 0$
and Lemma~\ref{C_2Z11_HerIso}
implies that 
${\overline{x A x}}$
is isomorphic to a hereditary subalgebra of ${\overline{z A z}}.$
Therefore ${\overline{x A_{\ld} x}}$
is isomorphic to a subalgebra of~$B.$
The hypotheses imply that ${\overline{x A x}} \in \CC.$
So $B \in \CC$ because $\CC$ is upwards directed.
This proves~(\ref{P-609_CCQHILim-HI}).

Now we prove~(\ref{P-609_CCQHILim-Q}).
Let $I \subset A$ be an ideal.
For $\ld \in \Ld,$
define $I_{\ld} = A_{\ld} / \ph_{\ld}^{-1} (I).$
Then $( A_{\ld} / I_{\ld} )_{\ld \in \Ld}$
is a direct system of \ca{s}
with injective maps and whose direct limit is $A / I.$
The algebras $A_{\ld} / I_{\ld}$ are hereditarily in~$\CC$
by definition,
so $A / I$ is hereditarily in~$\CC$ by~(\ref{P-609_CCQHILim-HI}).
Since $I$ is arbitrary,
this proves that $A$ is residually hereditarily in~$\CC.$
\end{proof}

\begin{prp}\label{P-609_CCQHIHered}
Let $\CC$ be an upwards directed class of \ca{s}.
Let $A$ be a \ca,
and let $B \subset A$ be a hereditary subalgebra.
\begin{enumerate}
\item\label{P-609_CCQHIHered-HI}
If $A$ is hereditarily in~$\CC,$
then $B$ is hereditarily in~$\CC.$
\item\label{P-609_CCQHIHered-Q}
If $A$ is residually hereditarily in~$\CC,$
then $B$ is residually hereditarily in~$\CC.$
\end{enumerate}
\end{prp}

\begin{proof}
Part~(\ref{P-609_CCQHIHered-HI}) is immediate from
the fact that hereditary subalgebras
of hereditary subalgebras
are hereditary.

We prove~(\ref{P-609_CCQHIHered-Q}).
So let $J \subset B$ be an ideal.
Theorem 3.2.7 of~\cite{Mr} provides an ideal $I \subset A$
such that $I \cap B = J.$
Let $\pi \colon A \to A / I$
be the quotient map.
Then the restriction to $B$ of $\pi$
induces an injective \hm\  $\ph \colon B / J \to A / I.$
To finish the proof,
by part~(\ref{P-609_CCQHIHered-HI}) it suffices to show that
$\ph (B / J)$ is a hereditary subalgebra of $A / I.$
Since $\ph (B / J) = \pi (B),$
this follows from Lemma~\ref{L_3112_ImHered}.
\end{proof}

\begin{prp}\label{P-609_CCQHIMn}
Let $A$ be a \ca,
and let $n \in \N.$
Then:
\begin{enumerate}
\item\label{P-609_CCQHIMn-HI}
$A$ is hereditarily in~$\CC$
\ifo\  $M_n (A)$ is hereditarily in~$\CC.$
\item\label{P-609_CCQHIMn-Q}
$A$ is residually hereditarily in~$\CC$
\ifo\  $M_n (A)$ is residually hereditarily in~$\CC.$
\end{enumerate}
\end{prp}

\begin{proof}
If $M_n (A)$ is (residually) hereditarily in~$\CC,$
then Proposition~\ref{P-609_CCQHIHered}
implies that $A$ is (residually) hereditarily in~$\CC.$

Now assume that $A$ is hereditarily in~$\CC.$
Let $B \subset M_n (A)$ be a nonzero hereditary subalgebra.
Choose $z \in B_{+} \setminus \{ 0 \}.$
Let $(e_{j, k})_{j, k = 1}^n$ be the standard system
of matrix units in~$M_n.$
For $k = 1, 2, \ldots, n,$
define $f_k = e_{k, k} \otimes 1 \in M_n (M (A)).$
Since $\sum_{k = 1}^n f_k = 1$ and $z^{1/2} \neq 0,$
there is $k$ such that $f_k z^{1/2} \neq 0.$
Set
\[
x = f_k z f_k \in f_k M_n (A) f_k \cong A
\andeqn
y = z^{1/2} f_k z^{1/2} \in B.
\]
Then ${\overline{x M_n (A) x}}$ is isomorphic to
a hereditary subalgebra of~$A,$
whence ${\overline{x M_n (A) x}} \in \CC.$
The last part of Lemma~\ref{L-Polar606} implies
that ${\overline{y M_n (A) y}} \cong {\overline{x M_n (A) x}}.$
Since ${\overline{y M_n (A) y}} \S B,$
we get $B \in \CC.$

Finally, assume that $A$ is residually hereditarily in~$\CC.$
Let $J \S M_n (A)$ be an ideal.
Then there is an ideal $I \S A$ such that $J = M_n (I).$
The hypotheses imply
that $A / I$ is hereditarily in~$\CC,$
so $M_n (A) / J$ is hereditarily in~$\CC$
by the previous paragraph.
\end{proof}

\begin{prp}\label{P-609_CCQHIStab}
Let $\CC$ be an upwards directed class of \ca{s}
and let $A$ be a \ca.
Then:
\begin{enumerate}
\item\label{P-609_CCQHIStab-HI}
$A$ is hereditarily in~$\CC$
\ifo\  $K \otimes A$ is hereditarily in~$\CC.$
\item\label{P-609_CCQHIStab-Q}
$A$ is residually hereditarily in~$\CC$
\ifo\  $K \otimes A$ is residually hereditarily in~$\CC.$
\end{enumerate}
\end{prp}

\begin{proof}
If $K \otimes A$ is (residually) hereditarily in~$\CC,$
then Proposition~\ref{P-609_CCQHIHered}
implies that $A$ is (residually) hereditarily in~$\CC.$
If $A$ is (residually) hereditarily in~$\CC,$
then apply Proposition~\ref{P-609_CCQHIMn}
and Proposition~\ref{P-609_CCQHILim}
to the relation $K \otimes A = \dirlim_n M_n (A)$
to see that $K \otimes A$ is (residually) hereditarily in~$\CC.$
\end{proof}

\begin{cor}\label{C-609_CCQHIMorita}
Let $\CC$ be an upwards directed class of \ca{s}
Let $A$ and $B$ be Morita equivalent separable \ca{s}.
Then:
\begin{enumerate}
\item\label{C-609_CCQHIMorita-HI}
$A$ is hereditarily in~$\CC$
\ifo\  $B$ is hereditarily in~$\CC.$
\item\label{C-609_CCQHIMorita-Q}
$A$ is residually hereditarily in~$\CC$
\ifo\  $B$ is residually hereditarily in~$\CC.$
\end{enumerate}
\end{cor}

\begin{proof}
Since Morita equivalence for separable \ca{s} implies
stable isomorphism (Theorem~1.2 of~\cite{BGR}),
the result is immediate from Proposition~\ref{P-609_CCQHIStab}.
\end{proof}

\section{Hereditary infiniteness}\label{Sec:PermHInf}

\indent
We do not know whether the crossed product
of a purely infinite C*-algebra
by a discrete group is again purely infinite,
even under extra conditions such as finiteness of the group
or spectral freeness of the action.
(We do have results under the additional assumption
of the ideal property; see Theorem~\ref{T_3720_PIIP}.)
We therefore consider two formally weaker conditions,
namely
residual hereditary  infiniteness
and residual hereditary proper infiniteness,
for which we can use the methods of Section~\ref{Sec_HSAPerm}
to obtain permanence results,
in particular for crossed products.
It is not known whether our properties are equivalent to
pure infiniteness.
For the weaker one, this is Question~4.8 of~\cite{KR},
which is one of the motivations for these properties.
It is also not known whether they are equivalent to each other.

No freeness condition should be necessary
for permanence results for crossed products.
After all, the crossed product of a purely infinite C*-algebra
by a trivial action
is expected to be again purely infinite.

\begin{dfn}\label{D-HI428Mod}
Let $A$ be a \ca.
We say that $A$ is
{\emph{hereditarily infinite}}
if for every nonzero hereditary subalgebra $B \subset A,$
there is an infinite positive element $a \in B$
in the sense of Definition~3.2 of~\cite{KR},
that is, there is $b \in A_{+} \setminus \{ 0 \}$
such that $a \oplus b \precsim a.$
We say that $A$ is
{\emph{residually hereditarily infinite}}
if $A / I$ is hereditarily infinite for every ideal $I$ in~$A.$
\end{dfn}

There is a possible alternate definition.

\begin{dfn}\label{D-HI428Other}
Let $A$ be a \ca.
We say that $A$ is
{\emph{hereditarily properly infinite}}
if for every nonzero hereditary subalgebra $B \subset A,$
there is a properly infinite positive element $a \in B$
in the sense of Definition~3.2 of~\cite{KR},
that is,
such that $a \neq 0$ and $a \oplus a \precsim a.$
We say that $A$ is
{\emph{residually hereditarily properly infinite}}
if $A / I$ is hereditarily properly infinite for every ideal $I$ in~$A.$
\end{dfn}

The zero \ca{} satisfies all the conditions
in Definition~\ref{D-HI428Mod} and Definition~\ref{D-HI428Other}.
This seems to be the convenient choice.

Pure infiniteness implies residual hereditary proper infiniteness,
by Theorems 4.16 and 4.19 of~\cite{KR}.
Clearly residual hereditary proper infiniteness
implies residual hereditary infiniteness.
Question~4.8 of~\cite{KR} asks whether
a residually hereditarily infinite C*-algebra
is necessarily purely infinite.
As far as we know, this question is still open.

As a further motivation,
we cite the following result,
which is the equivalence of conditions (ii) and~(iv)
of Proposition~2.11 of~\cite{PR}
(valid, as shown there, even when $A$ is not separable).

\begin{prp}[Proposition~2.11 of~\cite{PR}]\label{P-609IdealPI}
Let $A$ be a \ca.
Then $A$ is purely infinite and has the ideal property
\ifo\  for every ideal $I \subset A,$
every nonzero hereditary subalgebra of $A / I$
contains an infinite projection.
\end{prp}

\begin{lem}\label{L_3719_InfHer}
Let $A$ be a \ca,
let $B \S A$ be a \hsa,
and let $a \in B_{+}.$
Suppose that there is $x \in A_{+} \SM \{ 0 \}$
such that $a \oplus x \precsim a$ in $M_2 (A).$
Then there is $y \in B_{+} \SM \{ 0 \}$
such that $a \oplus y \precsim a$ in $M_2 (B).$
\end{lem}

\begin{proof}
Choose $\ep > 0$ such that $(x - \ep)_{+} \neq 0.$
Lemma~2.5(ii) of~\cite{KR} provides $v \in A$
such that $(x - \ep)_{+} = v^* a v.$
Set $d = a^{1/2} v v^* a^{1/2} \in B.$
Then the last part of
Lemma~\ref{L-Polar606}
provides an isomorphism
$\ph \colon {\overline{(x - \ep)_{+} A (x - \ep)_{+}}}
    \to {\overline{d A d}} \S B$
such that $\ph (z) \sim z$
for all $z \in {\overline{(x - \ep)_{+} A (x - \ep)_{+}}}.$
Set $y = \ph \big( (x - \ep)_{+} \big).$
Then, in $M_2 (A),$
we have
\[
a \oplus y
 \sim a \oplus (x - \ep)_{+}
 \leq a \oplus x
 \precsim a.
\]
Since $y \in B,$
it follows from Lemma 2.2(iii) of~\cite{KR}
that $a \oplus y \precsim a$ in $M_2 (B).$
\end{proof}

\begin{cor}\label{C_3719_HerInfClass}
Let $A$ be a \ca.
Let $\CC$ be the class of all \ca{s} which
contain an infinite element.
Then $\CC$ is upwards directed
and $A$ is (residually) hereditarily infinite
\ifo{} $A$ is (residually) hereditarily in~$\CC.$
\end{cor}

\begin{proof}
It is obvious that $\CC$ is upwards directed,
and the second part follows from Lemma~\ref{L_3719_InfHer}.
\end{proof}

We can now give the permanence theorems.
In both of them,
we have listed the crossed product results first.

\begin{thm}\label{T_3719_HIPr}
The following operations preserve hereditary (proper) infiniteness:
\begin{enumerate}
\item\label{T_3719_HIPr_SpNtr}
Reduced crossed products by spectrally nontrivial actions
of discrete groups.
\item\label{T_3719_HIPr_RkZ}
Crossed products by Rokhlin actions of~$\Z.$
\item\label{T_3719_HIPr_Z2FP}
Passage to a \ca~$A$ from the fixed point algebra under an action
of $\Z_2$:
if $\af \colon \Z_2 \to \Aut (A)$
is an arbitrary action of $\Z_2$ on~$A,$
and $A^{\af}$ is hereditarily (properly) infinite,
then so is~$A.$
\item\label{T_3719_HIPr_Z2Cr}
Crossed products by arbitrary actions of~$\Z_2.$
\item\label{T_3719_HIPr_her}
Passage to hereditary subalgebras.
\item\label{T_3719_HIPr_InjLim}
Direct limits (over arbitrary index sets)
of systems in which all the maps are injective.
\end{enumerate}
\end{thm}

\begin{thm}\label{T_3719_ResIP}
The following operations preserve
residual hereditary (proper) infiniteness:
\begin{enumerate}
\item\label{T_3719_ResIPr_SpFr}
Reduced crossed products by exact spectrally free actions
of discrete groups.
\item\label{T_3719_ResIPr_RkZ}
Crossed products by Rokhlin actions of~$\Z.$
\item\label{T_3719_ResIPr_Z2FP}
Passage to a \ca~$A$ from the fixed point algebra under an action
of $\Z_2$:
if $\af \colon \Z_2 \to \Aut (A)$
is an arbitrary action of $\Z_2$ on~$A,$
and $A^{\af}$ is residually hereditarily (properly) infinite,
then so is~$A.$
\item\label{T_3719_ResIPr_Z2Cr}
Crossed products by arbitrary actions of~$\Z_2.$
\item\label{T_3719_ResIPr_her}
Passage to hereditary subalgebras.
\item\label{T_3719_ResIPr_DLim}
Direct limits (over arbitrary index sets).
(The maps of the system need not be injective.)
\item\label{T_3719_ResIPr_2Of3}
Two out of three in short exact sequences:
if $A$ is a \ca,
and $J \subset A$ is an ideal,
then $A$ is residually hereditarily (properly) infinite
\ifo\  $J$ and $A / J$ are both
residually hereditarily (properly) infinite.
\item\label{T_3719_ResIPr_Stab}
Stable isomorphism,
equivalently,
$A$ is residually hereditarily (properly) infinite
\ifo{} $K \otimes A$ is residually hereditarily (properly) infinite.
\item\label{T_3719_ResIPr_Mor}
Morita equivalence for separable \ca{s}.
\end{enumerate}
\end{thm}

\begin{proof}[Proofs of Theorem~\ref{T_3719_HIPr}
   and Theorem~\ref{T_3719_ResIP}]
The statements for (residual) hereditary infiniteness
follow from Corollary~\ref{C_3719_HerInfClass}
and the results of Section~\ref{Sec_HSAPerm}.

For the statements for (residual) hereditary proper infiniteness,
let $\CC$ be the class of all \ca{s} which contain a properly
infinite element.
This class is obviously upwards directed,
and a \ca{} is (residually) hereditarily properly infinite
\ifo{} it is (residually) hereditarily in~$\CC$
by definition.
These statements therefore also follow from
the results of Section~\ref{Sec_HSAPerm}.
\end{proof}

Proposition~\ref{P-609IdealPI}
shows that
purely infinite \ca{s} with the ideal property
are also covered by our methods.
We thus have (Theorem~\ref{T_3720_PIIP} below)
the analog of Theorem~\ref{T_3719_ResIP}.
Before stating and proving it,
we make some comments on the parts.

Theorem \ref{T_3720_PIIP}(\ref{T_3720_PIIP_SpFr})
has been independently proved for partial actions
in Theorem~4.2 of~\cite{GrSr}.
For separable~$A,$
the it is essentially already in the literature.
First, apply Remark~2.5 of~\cite{Ks7}
to deduce that spectral freeness implies that,
for every $\af$-invariant ideal $I \subset A$ and every
$g \in G \setminus \{ 1 \},$
the induced automorphism of $A / I$ is properly outer.
Then use Remark~2.23 of~\cite{Sr}
(relying on Theorem~1.13 of~\cite{Sr};
the reference to Theorem~1.10 of~\cite{Sr} is a misprint)
to conclude that
$A$ separates the ideals in $C^*_{\mathrm{r}} (G, A, \af).$
This is enough to apply the proof of Lemma~3.1 of~\cite{RrSr},
and thus show that
$C^*_{\mathrm{r}} (G, A, \af)$ has the ideal property.
Proper outerness in place of essentially free action of $G$
on ${\widehat{A}}$
is also enough for the proof of Lemma~3.2 of~\cite{RrSr}
to be valid.
Combining this with the fact that
$A$ separates the ideals in $C^*_{\mathrm{r}} (G, A, \af),$
the proof of (i) implies~(iii) in Theorem~3.3 of~\cite{RrSr}
goes through,
and shows that $C^*_{\mathrm{r}} (G, A, \af)$ is purely infinite.
Separability enters because the proof of Remark~2.23 of~\cite{Sr}
relies on Lemma~7.1 of~\cite{OPd3},
which is only stated for separable \ca{s}.
We do not need separability for our version,
namely Lemma~\ref{L-L8_012Mod}.
Otherwise, our proof is fairly close.

On the other hand,
the results for actions of~$\Z_2,$
Theorem \ref{T_3720_PIIP}(\ref{T_3720_PIIP_Z2FP})
and Theorem \ref{T_3720_PIIP}(\ref{T_3720_PIIP_Z2Cr}),
seem to be completely new.

There are situations
(such as for minimal actions on infinite compact Hausdorff spaces)
in which one gets the ideal property for a crossed product
even without assuming it for the original algebra.
No such result can be expected here,
as can be seen by considering the tensor product
of a pointwise outer action on a purely infinite simple C*-algebra
with the trivial action on $C ([0, 1]).$

The permanence results in Theorem~\ref{T_3720_PIIP}
which don't involve group actions seem not to have
been previously published
(except that part of Theorem \ref{T_3720_PIIP}(\ref{T_3720_PIIP_2Of3})
is Corollary 4.4(ii) of~\cite{PR}),
but presumably could easily have been proved earlier.
With the additional assumptions of separability and nuclearity,
many of them are in Proposition 3.7(2) of~\cite{Psn2};
see Definition 3.6 and Proposition 3.7(1) of~\cite{Psn2}.
Theorem \ref{T_3720_PIIP}(\ref{T_3720_PIIP_Stab})
isn't valid for the ideal property by itself,
as one sees by using the algebra in Example~\ref{Ex_3719_WkNotInter}.
Also,
extensions need not preserve the ideal property;
see Theorem~5.1 of~\cite{Psn1}.

\begin{thm}\label{T_3720_PIIP}
The following operations preserve
the class of
purely infinite \ca{s} with the ideal property:
\begin{enumerate}
\item\label{T_3720_PIIP_SpFr}
Reduced crossed products by exact spectrally free actions
of discrete groups.
\item\label{T_3720_PIIP_RkZ}
Crossed products by Rokhlin actions of~$\Z.$
\item\label{T_3720_PIIP_Z2FP}
Passage to a \ca~$A$ from the fixed point algebra under an action
of $\Z_2$:
if $\af \colon \Z_2 \to \Aut (A)$
is an arbitrary action of $\Z_2$ on~$A,$
and $A^{\af}$ is purely infinite and has the ideal property,
then the same is true of~$A.$
\item\label{T_3720_PIIP_Z2Cr}
Crossed products by arbitrary actions of~$\Z_2.$
\item\label{T_3720_PIIP_her}
Passage to hereditary subalgebras.
\item\label{T_3720_PIIP_DLim}
Direct limits (over arbitrary index sets).
(The maps of the system need not be injective.)
\item\label{T_3720_PIIP_2Of3}
Two out of three in short exact sequences:
if $A$ is a \ca,
and $J \subset A$ is an ideal,
then $A$ is purely infinite and has the ideal property
\ifo\  $J$ and $A / J$ are both
purely infinite and both have the ideal property.
\item\label{T_3720_PIIP_Stab}
Stable isomorphism,
equivalently,
$A$ is purely infinite and has the ideal property
\ifo{} $K \otimes A$ is purely infinite and has the ideal property.
\item\label{T_3720_PIIP_Mor}
Morita equivalence for separable \ca{s}.
\end{enumerate}
\end{thm}

\begin{proof}
Let $\CC$ be the class of \ca{s}~$A$
which contain an infinite projection.
Then $\CC$ is clearly upwards directed.
The result therefore follows
from the results of Section~\ref{Sec_HSAPerm}.
\end{proof}

Recall (Remark 2.5(vi) of~\cite{BP})
that a \ca~$A$ is said to have topological dimension zero
if the topology of $\Prim (A)$
has a base consisting of compact open sets.
(We do not require that these sets be closed.)

\begin{cor}\label{C-SpFrTopZimZ}
Let $\af \colon G \to \Aut (A)$
be an exact and spectrally free action
of a discrete group $G$ on a separable \ca~$A.$
If $A$ has topological dimension zero,
then the same is true of $C^*_{\mathrm{r}} (G, A, \af).$
\end{cor}

\begin{proof}
It follows from Proposition~4.5 of~\cite{KR}
that ${\mathcal{O}}_2 \otimes A$
is purely infinite.
Now $\Prim ( {\mathcal{O}}_2 \otimes A) \cong \Prim (A),$
so ${\mathcal{O}}_2 \otimes A$ has topological dimension zero.
Using (i) implies~(ii) in Theorem~2.11 of~\cite{PR},
we deduce that ${\mathcal{O}}_2 \otimes A$ has the ideal property.
Let $\bt \colon G \to \Aut ({\mathcal{O}}_2 \otimes A)$
be the action $\bt_g = \id_{{\mathcal{O}}_2} \otimes \af_g$
for $g \in G.$
Proposition~\ref{P-610SfrTP}(\ref{P-610SfrTP-SFr})
implies that $\bt$ is spectrally free.
Theorem \ref{T_3720_PIIP}(\ref{T_3720_PIIP_SpFr}) now implies that
${\mathcal{O}}_2 \otimes C^*_{\mathrm{r}} (G, A, \af)
  \cong C^*_{\mathrm{r}}
          ( G, \, {\mathcal{O}}_2 \otimes A, \, \bt )$
is purely infinite and has the ideal property.
Using (ii) implies~(i) in Theorem~2.11 of~\cite{PR},
we deduce that ${\mathcal{O}}_2 \otimes C^*_{\mathrm{r}} (G, A, \af)$
has topological dimension zero.
Since
$\Prim \big( {\mathcal{O}}_2 \otimes C^*_{\mathrm{r}} (G, A, \af) \big)
  \cong \Prim \big( C^*_{\mathrm{r}} (G, A, \af) \big),$
it follows that $C^*_{\mathrm{r}} (G, A, \af)$
has topological dimension zero.
\end{proof}

In Corollary~\ref{C-SpFrTopZimZ},
separability should not be necessary.

\section{Property~(SP) and residual Property~(SP)}\label{Sec:PermQSP}

\indent
Recall that a \ca{} has Property~(SP) if every nonzero hereditary
subalgebra contains a \nzp.
This property is commonly used for simple \ca{s}.
For nonsimple \ca{s},
it seems more appropriate
to ask for the following strengthened version.

\begin{dfn}\label{D-QSP606}
Let $A$ be a \ca.
We say that $A$ has {\emph{residual~(SP)}}
if $A / I$ has Property~(SP) for every ideal $I \subset A.$
\end{dfn}

This property has already implicitly appeared
as a hypothesis in the literature;
see Proposition~\ref{P-609IdealPI}.

We prove permanence results for Property~(SP) and residual~(SP).
The results for crossed products by infinite groups
definitely require some sort of freeness condition,
as one can see by considering the trivial action of $\Z$ on~$\C.$

Property~(SP) does not imply residual~(SP).

\begin{exa}\label{E-QSPvsSP}
Let $H = l^2 (\Z)$ and let $D \subset L (H)$ be the \ca{} generated
by $K = K (H)$ and the bilateral shift.
We claim that $D$ has Property~(SP) but not residual~(SP).

That $D$ does not have residual~(SP) follows from the fact
that $D / K \cong C (S^1).$

We verify that $D$ has Property~(SP).
Let $B \subset D$ be a nonzero hereditary subalgebra.
We first claim that $B \cap K \neq \{ 0 \}.$
Choose $b \in B_{+} \setminus \{ 0 \}.$
Choose $\xi \in H$ such that $b \xi \neq 0.$
Let $p \in K$ be the projection on $\C \xi.$
Then $b p \xi \neq 0.$
Therefore $b p b = (b p) (b p)^*$ is a nonzero element of $B \cap K.$
This proves the claim.

Now $B \cap K$ is a nonzero hereditary subalgebra of~$K,$
so contains a \nzp, as desired.
\end{exa}

It follows from Theorem~4.2 of~\cite{JO} that
if $\af \colon G \to \Aut (A)$
is a pointwise outer action
of a discrete group $G$ on a simple \ca~$A$
with Property~(SP),
then $C^*_{\mathrm{r}} (G, A, \af)$ has Property~(SP).
(The actual statement has slightly weaker hypotheses:
one only requires that $\af_g$ be outer for $g$ outside
some finite normal subgroup of~$G.$)
Theorem \ref{T_3719_SPPr}(\ref{T_3719_SPPr_SpNtr})
and Theorem \ref{T_3719_ResidSP}(\ref{T_3719_ResidSP_SpFr}) below
give generalizations to nonsimple \ca{s}.

\begin{thm}\label{T_3719_SPPr}
The following operations preserve Property~(SP):
\begin{enumerate}
\item\label{T_3719_SPPr_SpNtr}
Reduced crossed products by spectrally nontrivial actions
of discrete groups.
\item\label{T_3719_SPPr_RkZ}
Crossed products by Rokhlin actions of~$\Z.$
\item\label{T_3719_SPPr_Z2FP}
Passage to a \ca~$A$ from the fixed point algebra under an action
of $\Z_2$:
if $\af \colon \Z_2 \to \Aut (A)$
is an arbitrary action of $\Z_2$ on~$A,$
and $A^{\af}$ has Property~(SP),
then so does~$A.$
\item\label{T_3719_SPPr_Z2Cr}
Crossed products by arbitrary actions of~$\Z_2.$
\item\label{T_3719_SPPr_her}
Passage to hereditary subalgebras.
\item\label{T_3719_SPPr_InjLim}
Direct limits (over arbitrary index sets)
of systems in which all the maps are injective.
\end{enumerate}
\end{thm}

\begin{thm}\label{T_3719_ResidSP}
The following operations preserve
residual~(SP):
\begin{enumerate}
\item\label{T_3719_ResidSP_SpFr}
Reduced crossed products by exact spectrally free actions
of discrete groups.
\item\label{T_3719_ResidSP_RkZ}
Crossed products by Rokhlin actions of~$\Z.$
\item\label{T_3719_ResidSP_Z2FP}
Passage to a \ca~$A$ from the fixed point algebra under an action
of $\Z_2$:
if $\af \colon \Z_2 \to \Aut (A)$
is an arbitrary action of $\Z_2$ on~$A,$
and $A^{\af}$ has residual~(SP),
then so does~$A.$
\item\label{T_3719_ResidSP_Z2Cr}
Crossed products by arbitrary actions of~$\Z_2.$
\item\label{T_3719_ResidSP_her}
Passage to hereditary subalgebras.
\item\label{T_3719_ResidSP_DLim}
Direct limits (over arbitrary index sets).
(The maps of the system need not be injective.)
\item\label{T_3719_ResidSP_2Of3}
Two out of three in short exact sequences:
if $A$ is a \ca,
and $J \subset A$ is an ideal,
then $A$ has residual~(SP)
\ifo\  $J$ and $A / J$ both
have residual~(SP).
\item\label{T_3719_ResidSP_Stab}
Stable isomorphism,
equivalently,
$A$ has residual~(SP)
\ifo{} $K \otimes A$ has residual~(SP).
\item\label{T_3719_ResidSP_Mor}
Morita equivalence for separable \ca{s}.
\end{enumerate}
\end{thm}

\begin{proof}[Proofs of Theorem~\ref{T_3719_SPPr}
   and Theorem~\ref{T_3719_ResidSP}]
Let $\CC$ be the class of all \ca{s} which contain a nonzero
projection.
This class is obviously upwards directed,
and by definition a \ca{} has (residual)~(SP)
\ifo{} it is (residually) hereditarily in~$\CC.$
All parts therefore follow from
the results of Section~\ref{Sec_HSAPerm}.
\end{proof}

Some condition on the action is needed to be able to prove
that a crossed product has Property~(SP),
as one can see by considering the trivial action of $\Z$ on~$\C.$

The following example shows that pointwise spectral nontriviality
(without requiring anything about the action on quotients)
is not a strong enough condition to prove preservation of
residual~(SP).

\begin{exa}\label{E-606CPNoQSP}
Let $G = \Z,$
let $X = \Z \cup \{ \infty \}$
be the one point compactification of~$\Z,$
and let $\Z$ act on $X$ by translation on~$\Z$
and with $\I$ fixed.
Set $A = C (X),$
and let $\af \colon \Z \to \Aut (A)$
be the corresponding action.
The algebra~$A$ clearly has residual~(SP).
The crossed product $C^* (\Z, A, \af)$ is isomorphic to
the algebra~$D$ of Example~\ref{E-QSPvsSP},
and thus does not have residual~(SP).

For $n \in \Z \setminus \{ 0 \},$
the set of fixed points for the action of $n$ on $X$ is $\{ \infty \},$
which has empty interior.
So $\af$ is pointwise spectrally nontrivial
by Lemma~\ref{L-608HmeSNT}.

Of course, it is easy to see that
$\af$ is not spectrally free,
by considering the quotient by the invariant ideal $C_0 (\Z).$
\end{exa}

\section{The weak ideal property}\label{Sec:WIP}

\indent
We do not know whether the crossed product of a \ca{} with the
ideal property by an arbitrary action of a finite group
again has the ideal property,
and this may be false.
However, a related but weaker property,
which we call the weak ideal property,
can be treated by the methods of this paper.
In particular, the weak ideal property
is preserved by crossed products by
exact spectrally free actions.
It is also preserved by arbitrary actions of finite abelian groups.
(We don't yet know about crossed products
by finite nonabelian groups.)
This weaker property also admits
better permanence results
of other kinds.
Here, like for~(SP), the trivial action of $\Z$ on~$\C$
shows that permanence results for crossed products by infinite groups
require a freeness condition.

\begin{dfn}\label{D-WIP2710}
Let $A$ be a \ca.
We say that $A$ has the {\emph{weak ideal property}}
if every nonzero subquotient of $K \otimes A$
contains a nonzero \pj.
\end{dfn}

An intermediate requirement is also possible:
one can require that every nonzero subquotient of $A$
contain a nonzero \pj.
We also consider this property,
although we do not give it a name.

\begin{prp}\label{L-IPWIP-2710}
Let $A$ be a \ca{} with the ideal property.
Then $A$ has the weak ideal property;
in fact, every nonzero subquotient of $A$
contains a nonzero \pj.
\end{prp}

\begin{proof}
Let $M \S L \S A$ be ideals with $M \neq L.$
Since the \pj s in $L$ generate $L$ as an ideal in~$A,$
there is a \pj\  $p \in L \SM M.$
Then $p + M$ is a \nzp\  in $L / M.$
\end{proof}

\begin{exa}\label{Ex_3719_WkNotInter}
Let $D$ be any infinite dimensional simple separable unital \ca{}
with no \pj{s} other than $0$ and~$1.$
(Example: The Jiang-Su algebra.)
Let $A \S D$ be any proper nonzero \hsa.
Then $A$ has the weak ideal property,
because $K \otimes A \cong K \otimes D$
(by Theorem~2.8 of~\cite{Brn}).
However, $A$ has no nonzero \pj{s},
so it is certainly not true that every nonzero subquotient of $A$
contain a nonzero \pj.
\end{exa}

\begin{exa}\label{Ex_3719_InterNotSt}
There is a separable \ca~$A$ with the property that
that every nonzero subquotient of $A$
contain a nonzero \pj,
but which does not have the ideal property.
Our example depends on the fact that
extensions of algebras with the ideal property
need not have the ideal property.
Let $A$ be the C*-algebra
constructed in the proof of Theorem~5.1 of~\cite{Psn1}.
There is a short exact sequence
\[
0 \longrightarrow I
  \longrightarrow A
  \longrightarrow \C
  \longrightarrow 0,
\]
in which $I$ is a stabilized Bunce-Deddens algebra,
and which does not split.
(In particular, $A$ is not unital.)
The only nonzero subquotients of $A$ are $A,$ $I,$ and $A / I \cong \C.$
All contain nonzero \pj s.
Thus $A$ has the weak ideal property.
However, by Theorem~5.1 of~\cite{Psn1},
the algebra $A$ does not have the ideal property.
\end{exa}

The methods of Section~\ref{Sec_HSAPerm}
give the following permanence results
for the weak ideal property.
We omit the results inolving actions of~$\Z_2,$
since we obtain better versions with a separate argument.

\begin{thm}\label{T_3719_PrWkId}
The following operations preserve
the weak ideal property:
\begin{enumerate}
\item\label{T_3719_PrWkId_SpFr}
Reduced crossed products by exact spectrally free actions
of discrete groups.
\item\label{T_3719_PrWkId_RkZ}
Crossed products by Rokhlin actions of~$\Z.$
\item\label{T_3719_PrWkId_her}
Passage to hereditary subalgebras.
\item\label{T_3719_PrWkId_DLim}
Direct limits (over arbitrary index sets).
(The maps of the system need not be injective.)
\item\label{T_3719_PrWkId_2Of3}
Two out of three in short exact sequences:
if $A$ is a \ca,
and $J \subset A$ is an ideal,
then $A$ has the weak ideal property
\ifo\  $J$ and $A / J$ both
have the weak ideal property.
\item\label{T_3719_PrWkId_Stab}
Stable isomorphism,
equivalently,
$A$ has the weak ideal property
\ifo{} $K \otimes A$ has the weak ideal property.
\item\label{T_3719_PrWkId_Mor}
Morita equivalence for separable \ca{s}.
\end{enumerate}
\end{thm}

\begin{proof}
Let $\CC$ be the class of all \ca{s}~$A$
such that $K \otimes A$ contains a nonzero \pj.
Clearly $\CC$ is upwards directed.

Let $A$ be a \ca.
We claim that $A$ has the weak ideal property
\ifo{} $A$ is residually hereditarily in~$\CC.$
It suffices to show that $A$ has the property that
every nonzero ideal in $K \otimes A$ contains a nonzero \pj{}
\ifo{} $A$ is hereditarily in~$\CC.$

Assume that $A$ is hereditarily in~$\CC,$
and let $J \S K \otimes A$ be a nonzero ideal.
Then there is a nonzero ideal $I \S A$
such that $J = K \otimes I.$
Since ideals are \hsa{s},
it follows from the definition of $\CC$
that $J$ contains a nonzero \pj.

Now assume that
every nonzero ideal in $K \otimes A$ contains a nonzero \pj.
Let $B \S A$ be a nonzero \hsa.
Let $I \S A$ be the ideal generated by~$B,$
and let $p \in K \otimes I$ be a nonzero \pj.
Then $p$ can be approximated arbitrarily well in norm
by finite sums of elementary tensors of the form $k \otimes a_1 b a_2,$
with $k \in K,$ $a_1, a_2 \in A,$ and $b \in B.$
Therefore there are countable subsets $S \S A$ and $T \S B$
such that
\begin{equation}\label{Eq_3719_pInIdeal}
p \in {\overline{\spn \big(
  \big\{ k \otimes a_1 b a_2 \colon
   {\mbox{$k \in K,$ $a_1, a_2 \in S,$ and $b \in T$}} \big\} \big) }}.
\end{equation}
Let $A_0 \S A$ be the C*-subalgebra of $A$ generated by~$S \cup T,$
which is a separable \ca.
Let $B_0 \S A_0 \cap B$
be the \hsa{} of~$A_0$ generated by~$T.$
Let $I_0 \S A_0 \cap I$ be the ideal in~$A_0$ generated by~$B_0.$
It follows from~(\ref{Eq_3719_pInIdeal})
that $p \in K \otimes I_0.$
Since $B_0$ is full in~$I_0,$
we have $K \otimes B_0 \cong K \otimes I_0$
by Theorem~2.8 of~\cite{Brn}.
So $K \otimes B_0$ contains a nonzero \pj.
Since $K \otimes B_0 \subset K \otimes B,$
it follows that $K \otimes B$ contains a nonzero \pj.
This completes the proof of the claim.

Given the claim,
the theorem follows immediately
from the results in Section~\ref{Sec_HSAPerm}.
\end{proof}

Some of these permanence results also hold for
the property that every nonzero subquotient of $A$
contains a nonzero \pj.
It follows from Example~\ref{Ex_3719_WkNotInter}
that hereditary subalgebras, stable isomorphism, and Morita equivalence
do not preserve this condition.
It is not hard to prove directly that
the analogs of Theorem \ref{T_3719_PrWkId}(\ref{T_3719_PrWkId_DLim})
and Theorem \ref{T_3719_PrWkId}(\ref{T_3719_PrWkId_2Of3})
are valid.
We state explicitly
the analog of Theorem \ref{T_3719_PrWkId}(\ref{T_3719_PrWkId_SpFr}).
The analog of Theorem \ref{T_3719_PrWkId}(\ref{T_3719_PrWkId_RkZ})
then follows from Proposition~\ref{P-611-RPInpSFr}.

\begin{prp}\label{P-715SWIP-SP}
Let $\af \colon G \to \Aut (A)$
be an exact spectrally free action
of a discrete group $G$ on a \ca~$A.$
Suppose that every nonzero subquotient of $A$
contains a nonzero \pj.
Then the same is true of $C^*_{\mathrm{r}} (G, A, \af).$
\end{prp}

\begin{proof}
In view of Proposition~\ref{P-SNAQ-Sep608},
this follows from Corollary~\ref{C_2Z01_3rd}.
\end{proof}

We now consider actions of finite groups.
We need several lemmas.

\begin{lem}\label{L_3721_DSum}
Let $\af \colon G \to \Aut (A)$
be an action
of a finite group $G$ on a \ca~$A,$
and let the notation related to Lemma 5.3.3 of~\cite{Ph1}
be as before Lemma \ref{L_3724_Comp}.
Let $S \in {\mathcal{S}}_G,$
and let $M \S I_S / I_S^{-}$ be a nonzero ideal.
(In particular, $I_S^{-} \neq I_S.$)
Then there is a nonzero subquotient $L$ of $A^{\af}$
and an injective \hm{} from $L$ to~$M.$
\end{lem}

\begin{proof}
Since $I_S^{-}$ is $\af$-invariant
(Lemma 5.3.3(1) of~\cite{Ph1})
and $I_S^{-} \S I_S$
(Lemma 5.3.3(2) of~\cite{Ph1}),
we can define
$B = A / I_S^{-}$
and $J = (I + I_S^{-}) / I_S^{-},$
and further let
$\bt \colon G \to \Aut (B)$
be the action induced by~$\af.$
Then $J_S^{-} = \{ 0 \}$
and $J_S = (I_S + I_S^{-}) / I_S^{-} = I_S / I_{S}^{-},$
so
\begin{align*}
(J_S \cap J) / (J_S^{-} \cap J)
& = J_S \cap J
  = [(I_S + I_S^{-}) \cap (I + I_S^{-})] / I_S^{-}
      \\
& = [(I_S \cap I) + I_S^{-})] / I_S^{-}
  \cong (I_S \cap I) / (I_S^{-} \cap I).
\end{align*}
Moreover,
$B^{\bt} \cong A^{\af} / (I_{ \{ 1 \} } \cap A^{\af})$
(Lemma~1.6 of~\cite{Ph-EqSj}).
So it suffices to prove that there in an ideal in $B^{\bt}$
which is isomorphic to a subalgebra of~$M.$

By Lemma 5.3.3(5) of~\cite{Ph1},
there is a subgroup $H \subset G,$
an $H$-invariant ideal $N \subset J_S,$
a system $R$ of left coset representatives for $H$ in~$G,$
and a subset $R_0 \subset R,$
such that we have internal direct sum decompositions
\[
J_S = \bigoplus_{g \in R} \bt_g (N)
\andeqn
J \cap J_S = \bigoplus_{g \in R_0} \bt_g (N).
\]
Since $J_S \neq \{ 0 \},$
we have $\bigcap_{h \in S} \bt_h (J) \neq \{ 0 \}.$
Therefore $R_0 \neq \varnothing$
and $N \neq \{ 0 \}.$

We claim that $(J_S)^{\bt} \cong N^{\bt |_H}.$
Define an injective \hm{} $\ps \colon N \to J_S$
by $\ps (x) = \big( \bt_k (x) \big)_{k \in R}.$
Temporarily fix $g \in G.$
There is a bijection $\sm \colon R \to R$
and a function $\et \colon R \to H$
such that $g k = \sm (k) \et (k)$ for all $k \in R.$
For $l \in R,$
we then get
$\bt_g \big( \big( \bt_k (x) \big)_{k \in R} \big)_{\sm (l)}
  = \bt_{\sm (l) \et (l)} (x).$
Therefore
\begin{equation}\label{Eq_3721_Str}
\bt_g \big( (x_k)_{k \in R} \big)
 = \big( \bt_{k \et ( \sm^{-1} (k))} (x) \big)_{k \in R}.
\end{equation}
It is now clear that if $x \in N^{\bt |_H},$
then $\bt_g (\ps (x)) = \ps (x).$
Now suppose that $x \in N$ and $\bt_g (\ps (x)) = \ps (x)$
for all $g \in G.$
Let $h \in H$; we show $\bt_h (x) = x.$
There is a unique element $k_0 \in R \cap H.$
Set $g = h k_0^{-1} \in H.$
Let $\sm$ and $\et$ be as above.
Then $\sm (k_0) = k_0,$ so $\sm^{-1} (k_0) = k_0,$
and $\et (k_0) = k_0^{-1} h.$
Taking $k = k_0$ in~(\ref{Eq_3721_Str})
gives $\bt_h (x) = x.$
So $\ps (x) \in (J_S)^{\bt}$ \ifo{} $x \in N^{\bt |_H},$
and the claim follows.

Since $M$ is an ideal in $\bigoplus_{g \in R} \bt_g (N),$
there are ideals $M_g \S N$ for $g \in \R$
such that $M = \bigoplus_{g \in R} \bt_g (M_g).$
Since $M \neq \{ 0 \},$
there is $g \in R$ such that $M_g \neq \{ 0 \}.$
Then $L = M_g^{\bt |_H}$ is a nonzero ideal in~$N,$
hence in~$B^{\bt},$
and $\bt_g |_L$ is an injective \hm{}
from $L$ to~$M.$
This completes the proof.
\end{proof}

\begin{lem}\label{L_3721_SbQt}
Let $\af \colon G \to \Aut (A)$ be an action
of a finite group~$G$ on a \ca~$A,$
and let $C$ be a nonzero subquotient of~$A.$
Then there exists a nonzero subquotient of~$A^{\af}$
which is isomorphic to a subalgebra of~$C.$
\end{lem}

\begin{proof}
Choose ideals $I, J \S A$ such that $I \S J$ and $J / I = C.$
We use Lemma 5.3.3 of~\cite{Ph1},
and we follow the notation before Lemma~\ref{L_3721_DSum}.

First assume $J \not\subset I_{ \{ 1 \} }.$
Then $J / ( J \cap I_{ \{ 1 \} })$ is a nonzero subquotient of~$A.$
Since $I \S I_{ \{ 1 \} }$
(Lemma 5.3.3(3) of~\cite{Ph1}),
it suffices to prove that
there is a nonzero subquotient of~$A^{\af}$
which is isomorphic to a subquotient of $J / ( J \cap I_{ \{ 1 \} }).$
Since $I_{ \{ 1 \} }$ is $\af$-invariant
(Lemma 5.3.3(1) of~\cite{Ph1}),
we simplify the notation by defining
$B = A / I_{ \{ 1 \} }$
and $M = (J + I_{ \{ 1 \} }) / I_{ \{ 1 \} },$
and letting $\bt \colon G \to \Aut (B)$
be the action induced by~$\af.$
Then $M \cong J / ( J \cap I_{ \{ 1 \} }),$
and $B^{\bt} \cong A^{\af} / (I_{ \{ 1 \} } \cap A^{\af})$
by Lemma~1.6 of~\cite{Ph-EqSj}.
Since
$M \subset M_{ \{ 1 \} }$
(Lemma 5.3.3(3) of~\cite{Ph1})
and $M \neq \{ 0 \},$
Lemma~\ref{L_3724_Comp}
provides $S \in {\mathcal{S}}_G$
such that
$(M_S \cap \{ 0 \}) + M_S^{-} \subsetneq (M_S \cap M) + M_S^{-}.$
Then $M_S^{-} \subsetneq (M_S \cap M) + M_S^{-},$
so $M_S^{-} \cap M \subsetneq M_S \cap M.$
Lemma~\ref{L_3721_DSum} provides a nonzero subquotient of~$B^{\bt}$
which is isomorphic to
a subalgebra of the subquotient $(M_S \cap M) / (M_S^{-} \cap M)$
of~$M.$
The conclusion follows in this case.

Now assume $J \subset I_{ \{ 1 \} }.$
Since $I \subsetneq J,$
Lemma~\ref{L_3724_Comp}
provides $S \in {\mathcal{S}}_G$
such that
$(I_S \cap I) + I_S^{-} \subsetneq (I_S \cap J) + I_S^{-}.$
By Lemma 5.3.3(5) of~\cite{Ph1},
there is a subgroup $H \subset G,$
an $H$-invariant ideal $N \subset J_S,$
a system $R$ of left coset representatives for $H$ in~$G,$
and a subset $R_0 \subset R,$
such that we have internal direct sum decompositions
\[
I_S / I_S^{-} = \bigoplus_{g \in R} \bt_g (N)
\andeqn
[(I \cap I_S) + I_S^{-}] / I_S^{-} = \bigoplus_{g \in R_0} \bt_g (N).
\]
Since $[ (I_S \cap J) + I_S^{-} ] / I_S^{-}$
is an ideal in $I_S / I_S^{-}$
which strictly contains $[ (I \cap I_S) + I_S^{-} ] / I_S^{-},$
there is a nonzero ideal
$M \S \bigoplus_{g \in R \SM R_0} \bt_g (N) \S I_S / I_S^{-}$
such that
\[
[(I_S \cap I) + I_S^{-}] / I_S^{-} \oplus M
  = [(I_S \cap J) + I_S^{-}] / I_S^{-}.
\]
Then Lemma~\ref{L_3721_DSum} provides a nonzero subquotient of~$B^{\bt}$
which is isomorphic to
a subalgebra of $M.$
Moreover,
$M$ is isomorphic to an ideal in
\begin{align*}
\big[ [(I_S \cap J) + I_S^{-}] / I_S^{-} \big] \big/
  \big[ [(I_S \cap I) + I_S^{-}] / I_S^{-} \big]
& \cong [(I_S \cap J) + I_S^{-}] / [(I_S \cap I) + I_S^{-}]
\\
& \cong (I_S \cap J) \big/
    \big[ [(I_S \cap I) + I_S^{-}] \cap I_S \cap J \big]
\\
& = (I_S \cap J) / (I_S \cap I),
\end{align*}
which is a subquotient of $J / I.$
This completes the proof.
\end{proof}

\begin{thm}\label{T_3726_FixWkId}
Let $\af \colon G \to \Aut (A)$ be an action
of a finite group~$G$ on a \ca~$A.$
If $A^{\af}$ has the weak ideal property,
or if every nonzero subquotient of $A^{\af}$
contains a nonzero \pj,
then the same is true of~$A.$
\end{thm}

\begin{proof}
For the condition
that every nonzero subquotient of $A$
contains a nonzero \pj,
the result is immediate from Lemma~\ref{L_3721_SbQt}.
The result for the weak ideal property follows by tensoring with~$K.$
\end{proof}

\begin{cor}\label{C_3726_CrPdWkId}
Let $\af \colon G \to \Aut (A)$ be an action
of a finite abelian group~$G$ on a \ca~$A.$
If $A$ has the weak ideal property,
or if every nonzero subquotient of $A$
contains a nonzero \pj,
then the same is true of $C^* (G, A, \af).$
\end{cor}

\begin{proof}
Apply Theorem~\ref{T_3726_FixWkId}
to the dual action
${\widehat{\af}} \colon {\widehat{G}} \to \Aut ( C^* (G, A, \af) ).$
\end{proof}

\begin{qst}\label{Q_3727_NonAb}
Does Corollary~\ref{C_3726_CrPdWkId}
hold for finite nonabelian groups?
\end{qst}

\end{document}